\documentclass[submission,copyright]{eptcs}
\providecommand{\event}{ACT2022}
\usepackage{underscore}

\usepackage{amsmath, amsthm,amsfonts,amssymb}
\usepackage{mathtools,mathdots,scalerel,mathrsfs,stmaryrd,cmll}
\usepackage[utf8]{inputenc}
\usepackage[T1]{fontenc}
\usepackage{tikz,xparse,tcolorbox}
\usepackage[noend]{algpseudocode}
\usepackage{prooftree, listings}
\usepackage{enumitem}
\usepackage[british]{babel}

\usetikzlibrary{arrows, decorations.markings, shapes.geometric, decorations.pathmorphing, decorations.pathreplacing, intersections, patterns,calc, backgrounds}

\tikzset{
	0c/.style={circle, draw, fill, inner sep=.7pt},
	1c/.style={->, shorten <=2pt, shorten >=2pt},
	1clong/.style={->},
	edge/.style={shorten <=2pt, shorten >=2pt},
	equal/.style={shorten <=2pt, shorten >=2pt, double},
	1cinc/.style={right hook->, shorten <=2pt, shorten >=2pt},
	1csurj/.style={->>, shorten <=2pt, shorten >=2pt},
	1cincl/.style={left hook->, shorten <=2pt, shorten >=2pt},
	1cloop/.style={arloop->, shorten <=2pt, shorten >=2pt},
	2c/.style={double distance=1.5pt, shorten <=6pt, shorten >=8pt, decoration={markings,mark=at position -6pt with {\arrow[scale=1.5]{>}}}, preaction={decorate}},
	follow/.style={->, >=stealth, ultra thick, shorten <=3pt, shorten >=3pt, color=gray!70},
	bar/.style={ultra thick, shorten <=3pt, shorten >=3pt, color=magenta},
	arlabel/.style={scale=.8},
}

\newenvironment{codecomm}{
	\begin{center}\small
	\begin{tcolorbox}[width=0.9\textwidth, boxrule=0pt, colback=gray!20, arc=0pt, left=0pt, top=0pt, bottom=0pt, right=0pt, boxsep=.6em, colframe=gray!20]
}{
	\end{tcolorbox}
	\end{center}
}

\usepackage{environ}
\RenewEnviron{codecomm}{}

\newenvironment{algor}{
	\begin{center}
	\begin{tcolorbox}[width=0.9\textwidth, boxrule=1pt, colback=white, arc=0pt, left=0pt, top=0pt, bottom=0pt, right=0pt, boxsep=.6em, colframe=black]
}{
	\end{tcolorbox}
	\end{center}
}

\newcommand\diagset{\data{DiagSet}}
\newcommand\syncat{\cat{Ctx}[\diagset]}

\newcommand\lmap[3]{\ell^{#3}_{{#1} \incl {#2}}}

\newcommand\rmap[3]{r^{#3}_{{#1} \incl {#2}}}

\newcommand\eqdef{\coloneqq}
\newcommand\nbd{\nobreakdash-\hspace{0pt}}
\newcommand\idd[1]{\mathrm{id}_{#1}}

\newcommand\restr[2]{{#1}{\raisebox{0pt}{$|_{#2}$}}}

\newcommand\incl{\hookrightarrow}
\newcommand\incliso{\stackrel{\sim}{\hookrightarrow}}

\newcommand\surj{\twoheadrightarrow}

\newcommand\slice[2]{{#1}/{\raisebox{-2pt}{$#2$}}}

\newcommand\opp[1]{{#1}^\mathrm{op}}

\newcommand\hasse[1]{\mathscr{H}#1}

\newcommand\dmn[1]{\mathrm{dim}(#1)}

\newcommand\clos[1]{\mathrm{cl}#1}
\NewDocumentCommand \bord{g g} {\IfNoValueTF{#2}{%
	\IfNoValueTF{#1}{\partial}{\partial_{#1}}}{\partial_{#1}^{#2}}}
\NewDocumentCommand \sbord{g g} {\IfNoValueTF{#2}{%
	\IfNoValueTF{#1}{\Delta}{\Delta_{#1}}}{\Delta_{#1}^{#2}}}
\newcommand\cp[1]{\,{\scriptstyle\#}_{#1}\,}

\newcommand\submol{\sqsubseteq}
\newcommand\maxd[2]{\mathscr{M}_{#1}#2}

\newcommand\celto{\Rightarrow}

\newcommand\atom{{\raisebox{-.02em}{%
\begin{tikzpicture}[baseline={(current bounding box.south)}]%
	\node[circle, draw, line width=.05em, inner sep=.25em] at (0,0) {};%
	\node[circle, fill, inner sep=.06em] at (0,0) {};%
	\node[circle, fill, inner sep=.06em] at (0,.35em) {};%
\end{tikzpicture}}}}

\newcommand\skel[2]{\sigma_{\leq #1}#2}

\newcommand\gen[1]{\mathscr{#1}}

\newcommand\intp[1]{\llbracket{#1}\rrbracket}

\newcommand\cat[1]{\mathbf{#1}}

\newcommand\psh[2]{\mathrm{PSh}_{#1}(#2)}

\newcommand\ogpos{\cat{ogPos}}

\newcommand\dgmset{\atom\cat{Set}}

\newcommand\dgmcpx{\atom\cat{Cpx}_\textit{fsb}}

\newcommand\code[1]{\texttt{#1}}
\newcommand\data[1]{\mathsf{#1}}

\newtheoremstyle{ittheorem}
  {\topsep}   
  {\topsep}   
  {\itshape}  
  {0pt}       
  {\itshape \bfseries} 
  { ---}         
  {5pt plus 1pt minus 1pt} 
  {}          

\newtheoremstyle{itdfn}
  {\topsep}   
  {\topsep}   
  {}  
  {0pt}       
  {\bfseries} 
  {.}         
  {5pt plus 1pt minus 1pt} 
  {}          

\makeatletter
  \renewcommand\@upn{\textit}
\makeatother

\lstdefinestyle{codestyle}{
	backgroundcolor=\color{gray!20},
   	keywordstyle=\color{magenta},
   	numberstyle=\tiny\color{gray},
	stringstyle=\color{blue},
   	basicstyle=\ttfamily\footnotesize,
   	breakatwhitespace=false,
	breaklines=true,
	captionpos=b,
	keepspaces=true,
	numbers=left,
	numbersep=5pt,
	showspaces=false,
	showstringspaces=false,
	showtabs=false,
	tabsize=2
}
\theoremstyle{ittheorem}
\newtheorem{thm}{Theorem}[section]
\newtheorem{prop}[thm]{Proposition}

\newtheorem{lem}[thm]{Lemma}

\theoremstyle{itdfn}
\newtheorem{dfn}[thm]{}
\theoremstyle{remark}
\newtheorem{impl}[thm]{Implementation}
\newtheorem{rmk}[thm]{Remark}
\newtheorem{comm}[thm]{Comment}
\newtheorem{exm}[thm]{Example}

\title{Data Structures for Topologically Sound \\ Higher-Dimensional Diagram Rewriting}
\author{Amar Hadzihasanovic
\institute{$^1$ Tallinn University of Technology \\ $^2$ Quantinuum, 17 Beaumont Street, Oxford, UK}
\email{amar@ioc.ee}
\and
Diana Kessler
\institute{Tallinn University of Technology}
\email{diana-maria.kessler@taltech.ee}
}
\def\titlerunning{Data Structures for Topologically Sound Higher-Dimensional Diagram Rewriting}
\def\authorrunning{A. Hadzihasanovic \& D. Kessler}
\begin{document}
\maketitle

\begin{abstract}
We present a computational implementation of diagrammatic sets, a model of higher-dimensional diagram rewriting that is ``topologically sound'': diagrams admit a functorial interpretation as homotopies in cell complexes.
This has potential applications both in the formalisation of higher algebra and category theory and in computational algebraic topology.
We describe data structures for well-formed shapes of diagrams of arbitrary dimensions and provide a solution to their isomorphism problem in time $O(n^3 \log n)$.
On top of this, we define a type theory for rewriting in diagrammatic sets and provide a semantic characterisation of its syntactic category.
All data structures and algorithms are implemented in the Python library \code{rewalt}, which also supports various visualisations of diagrams.
\end{abstract}

\section*{Introduction}

This article concerns the computational implementation of higher\nbd dimensional diagrams in the sense of higher category theory, and contains some first steps in the computational complexity theory of diagrammatic rewriting in arbitrary dimensions.

Higher\nbd dimensional rewriting, as emergent from the theory of polygraphs \cite{burroni1993higher} -- see \cite{guiraud2019rewriting} for a survey -- is founded on an interpretation of \emph{rewrites as directed homotopies}.
A particular aim of our work is provable \emph{topological soundness}, namely, the existence of a functorial interpretation of rewrite systems as cell complexes, and of rewrites as homotopies.
This ensures that our implementation of higher\nbd dimensional rewriting can act as a formal system for homotopical algebra and higher category theory in all generality.

With this aim, we turn to the \emph{diagrammatic set} model \cite{hadzihasanovic2020diagrammatic} developed by the first author as a combinatorial alternative to polygraphs.
Diagrammatic sets have a dual nature as higher\nbd dimensional rewrite systems and ``combinatorial directed cell complexes''.
They support a model of weak higher categories and, unlike polygraphs, are topologically sound.

Beside the formalisation of higher algebra and category theory, potential applications are manifold. 
\emph{String diagram rewriting}, which is a form of 3\nbd dimensional rewriting, is arguably the characteristic computational mechanism of applied category theory.
It has been suggested \cite{bonfante2009polygraphic} that even ``classical'' forms of rewriting are more faithfully represented as diagram rewriting: for example, term rewriting implemented as rewriting in monoidal categories with cartesian structure explicitates the ``hidden costs'' of copying and deleting terms. 
In these contexts, it is important to have a grasp on the computational complexity of the basic operations of diagram rewriting, to ensure that one's cost model for a machine operating by diagram rewriting is reasonable.

\emph{Via} topological soundness, we also envisage applications to computational algebraic topology.
Directedness of cells gives an \emph{algebraic} grip on their pasting, which lends itself better to computation.
Directed cell complexes are also equipped with an orientation on their cells, which makes them naturally suited to the computation of cellular homology.

\subsection*{Structure of the paper}
In Section \ref{sec:structures}, we present some basic data structures from the theory of diagrammatic sets, together with their formal encoding: in particular, \emph{oriented graded posets} which are used to encode shapes of diagrams.

In Section \ref{sec:uniquerep}, we focus on the implementation of \emph{regular molecules}, the inductive subclass of oriented graded posets corresponding to well-formed shapes of diagrams.
To construct regular molecules, we need to decide their isomorphism problem; for general oriented graded posets, this is equivalent to the graph isomorphism problem (Proposition \ref{prop:gicomplete}), not known to be in $\data{P}$.
Our main result is a solution to the isomorphism problem for regular molecules in time $O(n^3 \log n)$ (Theorem \ref{thm:complexity}), which also gives us a canonical form, hence a unique representation of shapes of diagrams.

In Section \ref{sec:diagrams}, we move on to the formalisation of diagrams and diagrammatic sets. 
We present this in the form of a type theory $\diagset$ living ``on top'' of our implementation of shapes of diagrams: the terms, corresponding to diagrams, are ``filtered by regular molecules''.
This allows us to define formal semantics and give a semantic characterisation of our formal system (Theorem \ref{thm:synchar}).

\subsection*{Related work}
A number of type theories for higher-categorical structures of arbitrary dimension have been defined in recent years: most notably, Finster and Mimram's $\data{CaTT}$ \cite{finster2017type}, implementing the Maltsiniotis model of weak higher categories \cite{benjamin2021globular}, together with its ``strictly associative'' \cite{finster2021associative} and ``strictly unital'' \cite{finster2020unital} variants; and the \emph{opetopic} type theories by Ho Thanh, Curien, and Mimram \cite{thanh2019sequent, curien2019syntactic}.

The former are not particularly concerned with diagram rewriting, and focus instead on the implementation of coherent globular composition; the link to our work is tenuous.
The latter have some commonality, albeit with a focus on a more restrictive class of shapes.
In fact, $\diagset$ takes some inspiration not from one of the published opetopic type theories, but from a privately communicated variant due to Curien, which similarly rests on a ``black-boxed'' implementation of opetopic shapes.

Most closely related is the work by Vicary, Bar, Dorn, and others on quasistrict \cite{bar2017data} and later associative \cite{dorn2018associative, reutter2019high} $n$\nbd categories, serving as the foundation of the $\data{homotopy.io}$ proof assistant.
While the aim is nearly the same, we believe that our framework has a number of advantages over associative $n$\nbd categories.

From a theoretical perspective, it is only conjectural that associative $n$\nbd categories, in general, are topologically sound or satisfy the homotopy hypothesis. 
They also currently lack connections with other models of higher categories and a clear functorial viewpoint.
On the other hand, diagrammatic sets are topologically sound, satisfy a version of the homotopy hypothesis, and support a model of weak higher categories with concrete functorial ties to well-established models.

From a user perspective, the main point of divergence is that diagrams in associative $n$\nbd categories have ``strict units'' but ``weak interchange'', while our diagrams have ``strict interchange'' but need weak units to model ``nullary'' inputs or outputs.
For rewrite systems with many ``nullary'' generators, associative $n$\nbd categories may have a practical advantage, while diagrammatic sets are otherwise favoured.

Finally, in associative $n$\nbd categories, diagram shapes are essentially descriptions of cubical tilings, and by lack of strict interchange, each rewrite gets by default its own ``layer'' in the tiling.
This makes it so a ``local'' rewrite on a portion of a diagram leads to an inefficient ``global'' duplication of information.
Our ``face poset'' representation of diagrams, on the other hand, allows local rewrites to stay local, which is more efficient and will be beneficial to the parallelisability of diagram rewriting.

\subsection*{Implementation}
All data structures, algorithms, and systems discussed in this article were implemented by the authors as part of a Python library for higher\nbd dimensional rewriting and algebra, called \code{rewalt}.\footnote{Code: \url{https://github.com/ahadziha/rewalt}. Documentation: \url{https://rewalt.readthedocs.io}.}
An example of \code{rewalt} code is included in Example \ref{exm:lunital}.
The library also supports various kinds of visualisation for diagrams, optionally in the form of TikZ output.
All the Hasse and string diagrams in this article were generated by \code{rewalt} and included here with no subsequent retouching.

\subsection*{Acknowledgements}
This work was supported by the ESF funded Estonian IT Academy research measure (project 2014-2020.4.05.19-0001) and by the Estonian Research Council grant PSG764.

\section{Basic data structures} \label{sec:structures}

\begin{dfn}
In the theory of diagrammatic sets, the shape of a pasting diagram is encoded by its \emph{face poset}, recording whether a cell is located in the boundary of another cell, together with \emph{orientation} data which specifies whether an $(n-1)$\nbd dimensional cell is in the \emph{input} or \emph{output} half of the boundary of an $n$\nbd dimensional cell.
We call the mathematical structure containing these data an \emph{oriented graded poset}.
This is essentially the same as what Steiner calls a \emph{directed precomplex} \cite{steiner1993algebra} and Forest an \emph{$\omega$\nbd hypergraph} \cite{forest2019unifying}.
\end{dfn}

\begin{dfn}[Graded poset]
Let $P$ be a finite poset with order relation $\leq$ and let $P_\bot$ be $P$ extended with a least element $\bot$.
We say that $P$ is \emph{graded} if, for all $x \in P$, all directed paths from $x$ to $\bot$ in the Hasse diagram $\hasse{P_\bot}$, with edges going from covering to covered elements, have the same length.
If this length is $n+1$, we let $\dmn{x} \eqdef n$ be the \emph{dimension} of $x$. We write $P_n$ for the subset of $n$\nbd dimensional elements of $P$.
\end{dfn}

\begin{dfn}[Oriented graded poset]
An \emph{orientation} on a finite poset $P$ is an edge-labelling of its Hasse diagram with values in $\{+,-\}$. An \emph{oriented graded poset} is a finite graded poset with an orientation.
\end{dfn}

\begin{impl}
If we linearly order the elements of an oriented graded poset in each dimension, each element $x$ is uniquely identified by a pair of integers $(n, k)$, where $n$ is the \emph{dimension} of $x$, and $k$ is the \emph{position} of $x$ in the linear ordering of $n$\nbd dimensional elements.

\begin{codecomm}
The class \code{rewalt.ogposets.El} is an overlay for pairs of integers with methods \code{dim}, \code{pos} returning the first and second projection. 
\end{codecomm}

\noindent We then represent an oriented graded poset as a pair $(\data{face\_data}, \data{coface\_data})$ of \emph{arrays of arrays of pairs of sets of integers}, where
\begin{enumerate}
	\item $j \in \data{face\_data}[n][k][i]$ if and only if $(n-1, j)$ is covered by $(n, k)$, and
	\item $j \in \data{coface\_data}[n][k][i]$ if and only if $(n+1, j)$ covers $(n, k)$
\end{enumerate}
with orientation $-$ ($i = 0$) or $+$ ($i = 1$). We may implement the sets of integers as sorted arrays, or another data type which supports binary search in logarithmic time. This defines a data type $\data{OgPoset}$.

This representation is essentially an adjacency list representation of the poset's Hasse diagram, with vertices separated according to their dimension, and incoming and outgoing edges separated according to their label. 
If $E_P$ is the set of edges of the Hasse diagram of $P$, the $\data{OgPoset}$ representation of $P$ takes space $O(|P| + |E_P|)$.

Storing both $\data{face\_data}$ and $\data{coface\_data}$ is redundant since these are uniquely determined by each other. 
However, most of the computations we need to perform on oriented graded posets require regular access both to faces (covered elements) and cofaces (covering elements) of a given element, so it is advantageous to be able to access them in constant time.

\begin{codecomm}
Objects of the class \code{rewalt.ogposets.OgPoset} are built from \code{face\_data}, \code{coface\_data} as described. The class has an alternative constructor \code{from\_face\_data(face_data)} which computes \code{coface\_data} from the single datum of \code{face\_data}.
\end{codecomm} 
\end{impl}

\begin{exm} \label{exm:whisker}
Consider a diagram formed of one 2\nbd cell with two input 1\nbd cells and a single output 1\nbd cell, whiskered to the right with a single 1\nbd cell. The following are representations of its shape as
\begin{itemize}
	\item an oriented face poset, pictured as a Hasse diagram with input faces pointing upwards (in magenta) and output faces downwards (in blue);
	\item a string diagram (0-cells are unlabelled, but correspond to bounded regions of the plane);
	\item the pair of $\data{face\_data}$ and $\data{coface\_data}$ (rows are outer array indices and columns inner array indices).
\end{itemize}
\vspace{-5pt}
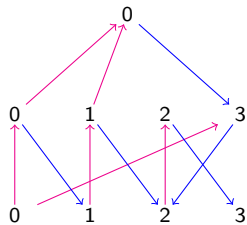
\begin{center}\footnotesize
	\begin{tikzpicture}[baseline={([yshift=-.5ex]current bounding box.center)}, xscale=4, yscale=4]
\path[fill=white] (0, 0) rectangle (1, 1);
\draw[->, draw=magenta] (0.125, 0.19999999999999998) -- (0.125, 0.4666666666666667);
\draw[->, draw=magenta] (0.19999999999999998, 0.19999999999999998) -- (0.8, 0.4666666666666667);
\draw[->, draw=magenta] (0.375, 0.19999999999999998) -- (0.375, 0.4666666666666667);
\draw[->, draw=magenta] (0.625, 0.19999999999999998) -- (0.625, 0.4666666666666667);
\draw[->, draw=blue] (0.15, 0.4666666666666667) -- (0.35, 0.19999999999999998);
\draw[->, draw=magenta] (0.16249999999999998, 0.5333333333333333) -- (0.4625, 0.7999999999999999);
\draw[->, draw=blue] (0.4, 0.4666666666666667) -- (0.6, 0.19999999999999998);
\draw[->, draw=magenta] (0.3875, 0.5333333333333333) -- (0.4875, 0.7999999999999999);
\draw[->, draw=blue] (0.65, 0.4666666666666667) -- (0.85, 0.19999999999999998);
\draw[->, draw=blue] (0.85, 0.4666666666666667) -- (0.65, 0.19999999999999998);
\draw[->, draw=blue] (0.5375, 0.7999999999999999) -- (0.8375, 0.5333333333333333);
\node[text=black, font={\scriptsize \sffamily}, xshift=0pt, yshift=0pt] at (0.125, 0.16666666666666666) {0};
\node[text=black, font={\scriptsize \sffamily}, xshift=0pt, yshift=0pt] at (0.375, 0.16666666666666666) {1};
\node[text=black, font={\scriptsize \sffamily}, xshift=0pt, yshift=0pt] at (0.625, 0.16666666666666666) {2};
\node[text=black, font={\scriptsize \sffamily}, xshift=0pt, yshift=0pt] at (0.875, 0.16666666666666666) {3};
\node[text=black, font={\scriptsize \sffamily}, xshift=0pt, yshift=0pt] at (0.125, 0.5) {0};
\node[text=black, font={\scriptsize \sffamily}, xshift=0pt, yshift=0pt] at (0.375, 0.5) {1};
\node[text=black, font={\scriptsize \sffamily}, xshift=0pt, yshift=0pt] at (0.625, 0.5) {2};
\node[text=black, font={\scriptsize \sffamily}, xshift=0pt, yshift=0pt] at (0.875, 0.5) {3};
\node[text=black, font={\scriptsize \sffamily}, xshift=0pt, yshift=0pt] at (0.5, 0.8333333333333333) {0};
\end{tikzpicture}
	\quad \quad
	
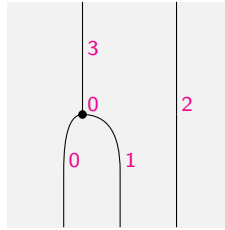
\begin{tikzpicture}[xscale=3, yscale=3, baseline={([yshift=-.5ex]current bounding box.center)}]
\path[fill=gray!10] (0, 0) rectangle (1, 1);
\draw[color=black, opacity=1] (0.3333333333333333, 0.5) .. controls (0.3333333333333333, 0.5) and (0.3333333333333333, 0.5833333333333333) .. (0.3333333333333333, 0.75);
\draw[color=black, opacity=1] (0.3333333333333333, 1) .. controls (0.3333333333333333, 1.0) and (0.3333333333333333, 0.9166666666666666) .. (0.3333333333333333, 0.75);
\draw[color=black, opacity=1] (0.75, 0) .. controls (0.75, 0.0) and (0.75, 0.16666666666666666) .. (0.75, 0.5);
\draw[color=black, opacity=1] (0.75, 1) .. controls (0.75, 1.0) and (0.75, 0.8333333333333333) .. (0.75, 0.5);
\draw[color=black, opacity=1] (0.3333333333333333, 0.5) .. controls (0.4444444444444444, 0.5) and (0.5, 0.41666666666666663) .. (0.5, 0.25);
\draw[color=black, opacity=1] (0.5, 0) .. controls (0.5, 0.0) and (0.5, 0.08333333333333333) .. (0.5, 0.25);
\draw[color=black, opacity=1] (0.3333333333333333, 0.5) .. controls (0.2777777777777778, 0.5) and (0.25, 0.41666666666666663) .. (0.25, 0.25);
\draw[color=black, opacity=1] (0.25, 0) .. controls (0.25, 0.0) and (0.25, 0.08333333333333333) .. (0.25, 0.25);
\node[circle, fill=black, draw=black, inner sep=1pt] at (0.3333333333333333, 0.5) {};
\node[text=magenta, font={\scriptsize \sffamily}, xshift=4pt, yshift=4pt] at (0.3333333333333333, 0.75) {3};
\node[text=magenta, font={\scriptsize \sffamily}, xshift=4pt, yshift=4pt] at (0.75, 0.5) {2};
\node[text=magenta, font={\scriptsize \sffamily}, xshift=4pt, yshift=4pt] at (0.5, 0.25) {1};
\node[text=magenta, font={\scriptsize \sffamily}, xshift=4pt, yshift=4pt] at (0.25, 0.25) {0};
\node[text=magenta, font={\scriptsize \sffamily}, xshift=4pt, yshift=4pt] at (0.3333333333333333, 0.5) {0};
\end{tikzpicture}
	\quad \quad
	\
	\begin{tabular}{l l l l}
		\multicolumn{4}{l}{$\data{face\_data}$:} \\
		$([], [])$ & $([], [])$ & $([], [])$ & $([], [])$ \\
		$([0], [1])$ & $([1], [2])$ & $([2], [3])$ & $([0], [2])$ \\
		$([0, 1], [3])$ & & & \\
		\multicolumn{4}{l}{$\data{coface\_data}$:} \\
		$([0, 3], [])$ & $([1], [0])$ & $([2], [1, 3])$ & $([], [2])$ \\
		$([0], [])$ & $([0], [])$ & $([], [])$ & $([], [0])$ \\
		$([], [])$ & & &	
	\end{tabular}
\end{center}
\end{exm}

\begin{rmk}
The representation of an oriented graded poset (up to isomorphism) is not unique: any permutation of the linear order on elements in each dimension leads to an equivalent representation.
\end{rmk}

\begin{dfn}
Many important computations are performed on \emph{(downwards) closed subsets}, rather than the whole of an oriented graded poset. 
In particular, the structure of an oriented graded poset supports a purely combinatorial definition of the input and output boundary of a closed subset.
\end{dfn}

\begin{dfn}[Closed subsets]
Let $P$ be an oriented graded poset and $U \subseteq P$. 
We say that $U$ is \emph{closed} if, for all $y \in U$ and $x \in P$, if $x \leq y$ then $x \in U$. The \emph{closure} of $U$ is the subset $\clos{U} \eqdef \{x \in P \mid \exists y \in U \; x \leq y\}$.

We let $\dmn{U}$ be the maximum of $\dmn{x}$ for $x \in U$, or $-1$ if $U$ is empty.
\end{dfn}

\begin{dfn}[Input and output boundaries]
Let $P$ be an oriented graded poset and $U \subseteq P$ a closed subset.
For all $\alpha \in \{+,-\}$ and $n \in \mathbb{N}$, let 
\begin{itemize}
	\item $\sbord{n}{\alpha} U \subseteq U$ be the subset of elements $x$ such that $\dmn{x} = n$ and, if $y \in U$ covers $x$, then it covers it with orientation $\alpha$;
	\item $\maxd{n}{U} \subseteq U$ be the subset of elements $x$ such that $\dmn{x} = n$ and $x$ is maximal in $U$ (not covered by any other element of $U$).
\end{itemize}
The \emph{input} ($\alpha \eqdef -$) or \emph{output} ($\alpha \eqdef +$) \emph{$n$\nbd boundary} of $U$ is the closed subset
\begin{equation*}
	\bord{n}{\alpha} U \eqdef \clos{\Big( \sbord{n}{\alpha} U \cup \bigcup_{k < n} \maxd{k}{U}\Big)}.
\end{equation*}
We let $\bord{n}{} U \eqdef \bord{n}{+}U \cup \bord{n}{-}U$ and omit $n$ when $n = \dmn{U} - 1$. For all $x \in P$, we let $\bord{n}{\alpha} x \eqdef \bord{n}{\alpha} \clos\{x\}$.
\end{dfn}

\begin{rmk}
It is convenient to also let $\bord{-1}{\alpha} U = \bord{-2}{\alpha} U \eqdef \emptyset$, so that $\bord{}{\alpha} U$ is defined for all $U \subseteq P$.
\end{rmk}

\begin{exm}
Let $U$ be the oriented face poset of Example \ref{exm:whisker}. Then
\begin{align*}
	\bord{1}{-}U & = \{(0, 0), (0, 1), (0, 2), (0, 3), (1, 0), (1, 1), (1, 2)\}, \\
	\bord{1}{+}U & = \{(0, 0), (0, 2), (0, 3), (1, 2), (1, 3)\}, \\
	\bord{0}{-}U & = \{(0, 0)\}, \quad \quad \bord{0}{+}U = \{(0, 3)\}.
\end{align*}
\end{exm}

\begin{impl}
We represent a set of elements of an $\data{OgPoset}$ as an \emph{array of sets of positions, indexed by dimensions}.
This allows us to access the subset of elements of a given dimension in constant time.
The size of arrays can be fixed to be equal to the dimension of a specific $\data{OgPoset}$, or dynamically adjusted to the dimension of each set of elements.
Sets of positions can again be implemented as sorted arrays.
This defines a data type $\data{GrSet}$ (for \emph{graded set}).

\begin{codecomm}
Graded sets are implemented in the class \code{rewalt.ogposets.GrSet}, which supports most methods of the in-built Python \code{set} class.

An object of the class \code{rewalt.ogposets.GrSubset} is a \code{GrSet}, the \code{support}, together with a pointer to an \code{OgPoset}, the \code{ambient}; upon initialisation, it is checked that all elements of the graded set are within range of the \code{OgPoset}.
 
The subclass \code{rewalt.ogposets.Closed}, which also checks for closedness, is equipped with methods \code{boundary([sign, n])} for computing input and output boundaries.
\end{codecomm}
\end{impl}

\begin{dfn}[Map of oriented graded posets]
A \emph{map} $f\colon P \to Q$ of oriented graded posets is a function of their underlying sets that satisfies	$\bord{n}{\alpha}f(x) = f(\bord{n}{\alpha}x)$ for all $x \in P$, $n \in \mathbb{N}$, and $\alpha \in \{+,-\}$. We call an injective map an \emph{inclusion}. Oriented graded posets and their maps form a category $\ogpos$.
\end{dfn}

\begin{exm}
A closed subset of an oriented graded poset inherits the structure of an oriented graded poset by restriction. Its subset inclusion is an inclusion of oriented graded posets. 
\end{exm}

\begin{impl}
We represent a map $f\colon P \to Q$ as an \emph{array of arrays of pairs of integers} $\data{mapping}$, together with pointers $\data{source}, \data{target}$ to $\data{OgPoset}$ representations of $P$ and $Q$. This defines a data type $\data{OgMap}$.
As an array of arrays, $\data{mapping}$ has the same size of $P$'s $\data{face\_data}$, and is defined by
\begin{center}
	$\data{mapping}[n][k] = (m, j)$ if and only if $f((n, k)) = (m, j)$.
\end{center}
This representation takes space $O(|P|)$.

\begin{codecomm}
The class \code{rewalt.ogposets.OgMap} also supports \emph{partial} maps, where \code{mapping[n][k]} may be set to \code{None}. When defining an \code{OgMap} on an element, a check is performed by default that the map is defined and compatible on all boundaries of that element. Composition of maps $f, g$ is implemented with syntax \code{f.then(g)}.

The class \code{rewalt.ogposets.Closed} has a method \code{as_map} which returns the inclusion of the closed subset as an \code{OgMap}.
\end{codecomm}
\end{impl}

\section{Unique representation of shapes of diagrams} \label{sec:uniquerep}

\begin{dfn}
In the theory of diagrammatic sets, shapes of diagrams form an inductively generated class of oriented graded posets, called regular \emph{molecules} after Steiner \cite{steiner1993algebra}.
\end{dfn}

\begin{dfn}[Round subset]
Let $U$ be a closed subset of an oriented graded poset, $n \eqdef \dmn{U}$.
We say that $U$ is \emph{round} if, for all $k < n$,
\begin{equation*}
	\bord{k}{+}U \cap \bord{k}{-}U = \bord{k-1}{}U.
\end{equation*}
\end{dfn}

\begin{rmk}
Roundness is called ``spherical boundary'' in \cite{hadzihasanovic2020diagrammatic}.
\end{rmk}

\begin{exm} \label{exm:frob}
Shapes of 2\nbd dimensional diagrams, as oriented face posets, are round precisely when
\begin{enumerate}
	\item their string diagram representation is \emph{connected}, and
	\item all nodes of the string diagram have at least one input and one output wire.
\end{enumerate}
For example, the oriented graded poset of Example \ref{exm:whisker} is not round: we have 
\begin{equation*}
	\bord{0}{}U = \{(0, 0), (0, 3)\} \subsetneq \bord{1}{+}U \cap \bord{1}{-}U = \{(0, 0), (0, 2), (0, 3)\}.
\end{equation*}
On the other hand, the following oriented graded poset is round:
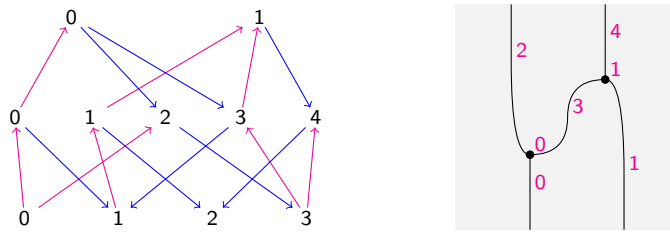
\begin{center}
	\begin{tikzpicture}[xscale=5, yscale=4, baseline={([yshift=-.5ex]current bounding box.center)}]
\path[fill=white] (0, 0) rectangle (1, 1);
\draw[->, draw=magenta] (0.1225, 0.19999999999999998) -- (0.10250000000000001, 0.4666666666666667);
\draw[->, draw=magenta] (0.16249999999999998, 0.19999999999999998) -- (0.4625, 0.4666666666666667);
\draw[->, draw=magenta] (0.3675, 0.19999999999999998) -- (0.30750000000000005, 0.4666666666666667);
\draw[->, draw=magenta] (0.8575, 0.19999999999999998) -- (0.7175, 0.4666666666666667);
\draw[->, draw=magenta] (0.8775, 0.19999999999999998) -- (0.8975000000000001, 0.4666666666666667);
\draw[->, draw=blue] (0.1275, 0.4666666666666667) -- (0.34750000000000003, 0.19999999999999998);
\draw[->, draw=magenta] (0.115, 0.5333333333333333) -- (0.23500000000000001, 0.7999999999999999);
\draw[->, draw=blue] (0.3325, 0.4666666666666667) -- (0.5925, 0.19999999999999998);
\draw[->, draw=magenta] (0.34500000000000003, 0.5333333333333333) -- (0.7050000000000001, 0.7999999999999999);
\draw[->, draw=blue] (0.5375, 0.4666666666666667) -- (0.8375, 0.19999999999999998);
\draw[->, draw=blue] (0.6675000000000001, 0.4666666666666667) -- (0.4075, 0.19999999999999998);
\draw[->, draw=magenta] (0.7050000000000001, 0.5333333333333333) -- (0.745, 0.7999999999999999);
\draw[->, draw=blue] (0.8725, 0.4666666666666667) -- (0.6525, 0.19999999999999998);
\draw[->, draw=blue] (0.275, 0.7999999999999999) -- (0.475, 0.5333333333333333);
\draw[->, draw=blue] (0.295, 0.7999999999999999) -- (0.655, 0.5333333333333333);
\draw[->, draw=blue] (0.765, 0.7999999999999999) -- (0.885, 0.5333333333333333);
\node[text=black, font={\scriptsize \sffamily}, xshift=0pt, yshift=0pt] at (0.125, 0.16666666666666666) {0};
\node[text=black, font={\scriptsize \sffamily}, xshift=0pt, yshift=0pt] at (0.375, 0.16666666666666666) {1};
\node[text=black, font={\scriptsize \sffamily}, xshift=0pt, yshift=0pt] at (0.625, 0.16666666666666666) {2};
\node[text=black, font={\scriptsize \sffamily}, xshift=0pt, yshift=0pt] at (0.875, 0.16666666666666666) {3};
\node[text=black, font={\scriptsize \sffamily}, xshift=0pt, yshift=0pt] at (0.1, 0.5) {0};
\node[text=black, font={\scriptsize \sffamily}, xshift=0pt, yshift=0pt] at (0.30000000000000004, 0.5) {1};
\node[text=black, font={\scriptsize \sffamily}, xshift=0pt, yshift=0pt] at (0.5, 0.5) {2};
\node[text=black, font={\scriptsize \sffamily}, xshift=0pt, yshift=0pt] at (0.7000000000000001, 0.5) {3};
\node[text=black, font={\scriptsize \sffamily}, xshift=0pt, yshift=0pt] at (0.9, 0.5) {4};
\node[text=black, font={\scriptsize \sffamily}, xshift=0pt, yshift=0pt] at (0.25, 0.8333333333333333) {0};
\node[text=black, font={\scriptsize \sffamily}, xshift=0pt, yshift=0pt] at (0.75, 0.8333333333333333) {1};
\end{tikzpicture}
	\quad \quad \quad
	\begin{tikzpicture}[xscale=3, yscale=3, baseline={([yshift=-.5ex]current bounding box.center)}]
\path[fill=gray!10] (0, 0) rectangle (1, 1);
\draw[color=black, opacity=1] (0.6666666666666666, 0.6666666666666666) .. controls (0.6666666666666666, 0.6666666666666666) and (0.6666666666666666, 0.7222222222222222) .. (0.6666666666666666, 0.8333333333333334);
\draw[color=black, opacity=1] (0.6666666666666666, 1) .. controls (0.6666666666666666, 1.0) and (0.6666666666666666, 0.9444444444444444) .. (0.6666666666666666, 0.8333333333333334);
\draw[color=black, opacity=1] (0.3333333333333333, 0.3333333333333333) .. controls (0.4444444444444444, 0.3333333333333333) and (0.5, 0.38888888888888884) .. (0.5, 0.5);
\draw[color=black, opacity=1] (0.6666666666666666, 0.6666666666666666) .. controls (0.5555555555555556, 0.6666666666666666) and (0.5, 0.611111111111111) .. (0.5, 0.5);
\draw[color=black, opacity=1] (0.3333333333333333, 0.3333333333333333) .. controls (0.2777777777777778, 0.3333333333333333) and (0.25, 0.4722222222222222) .. (0.25, 0.75);
\draw[color=black, opacity=1] (0.25, 1) .. controls (0.25, 1.0) and (0.25, 0.9166666666666666) .. (0.25, 0.75);
\draw[color=black, opacity=1] (0.6666666666666666, 0.6666666666666666) .. controls (0.7222222222222222, 0.6666666666666666) and (0.75, 0.5277777777777778) .. (0.75, 0.25);
\draw[color=black, opacity=1] (0.75, 0) .. controls (0.75, 0.0) and (0.75, 0.08333333333333333) .. (0.75, 0.25);
\draw[color=black, opacity=1] (0.3333333333333333, 0.3333333333333333) .. controls (0.3333333333333333, 0.3333333333333333) and (0.3333333333333333, 0.2777777777777778) .. (0.3333333333333333, 0.16666666666666666);
\draw[color=black, opacity=1] (0.3333333333333333, 0) .. controls (0.3333333333333333, 0.0) and (0.3333333333333333, 0.05555555555555555) .. (0.3333333333333333, 0.16666666666666666);
\node[circle, fill=black, draw=black, inner sep=1pt] at (0.3333333333333333, 0.3333333333333333) {};
\node[circle, fill=black, draw=black, inner sep=1pt] at (0.6666666666666666, 0.6666666666666666) {};
\node[text=magenta, font={\scriptsize \sffamily}, xshift=4pt, yshift=4pt] at (0.6666666666666666, 0.8333333333333334) {4};
\node[text=magenta, font={\scriptsize \sffamily}, xshift=4pt, yshift=4pt] at (0.5, 0.5) {3};
\node[text=magenta, font={\scriptsize \sffamily}, xshift=4pt, yshift=4pt] at (0.25, 0.75) {2};
\node[text=magenta, font={\scriptsize \sffamily}, xshift=4pt, yshift=4pt] at (0.75, 0.25) {1};
\node[text=magenta, font={\scriptsize \sffamily}, xshift=4pt, yshift=4pt] at (0.3333333333333333, 0.16666666666666666) {0};
\node[text=magenta, font={\scriptsize \sffamily}, xshift=4pt, yshift=4pt] at (0.3333333333333333, 0.3333333333333333) {0};
\node[text=magenta, font={\scriptsize \sffamily}, xshift=4pt, yshift=4pt] at (0.6666666666666666, 0.6666666666666666) {1};
\end{tikzpicture}
\end{center}
\end{exm}

\begin{dfn}[Regular molecules] \label{dfn:molecules}
The class of regular molecules is generated by the following clauses.
\begin{itemize}
\item (Point). The terminal oriented graded poset $\bullet$ is a regular molecule.
\item (Atom). Let $U, V$ be \emph{round} regular molecules such that $\dmn{U} = \dmn{V}$ and, for all $\alpha \in \{+, -\}$, $\bord{}{\alpha}U$ is isomorphic to $\bord{}{\alpha}V$.
	Then $U \celto V$ is a regular molecule, where $U \celto V$ is the essentially unique oriented graded poset $U \celto V$ with the property that
	\begin{enumerate}
	\item $U \celto V$ has a greatest element, and
	\item $\bord{}{-}(U \celto V)$ is isomorphic to $U$, while $\bord{}{+}(U \celto V)$ is isomorphic to $V$.
	\end{enumerate}
\item (Paste). Let $U, V$ be regular molecules and $k < \min (\dmn{U}, \dmn{V})$, such that $\bord{k}{+}U$ is isomorphic to $\bord{k}{-}V$.
Then the pushout $U \cp{k} V$ of the span $\bord{k}{+}U \incl U$, $\bord{k}{+}U \incliso \bord{k}{-}V \incl V$ is a regular molecule.
\end{itemize}
A regular molecule is an \emph{atom} if it has a greatest element; these are precisely the molecules whose final generating clause is (Point) or (Atom).

The \emph{submolecule} relation $U \submol V$ is the preorder generated by $U, V \submol U \celto V$ and $U, V \submol U \cp{k} V$.
\end{dfn}

\begin{comm}
The properties of regular molecules are explored in \cite[Sections 1, 2]{hadzihasanovic2020diagrammatic}.
Importantly, the following results ensure that \S \ref{dfn:molecules} is a valid definition:
\begin{enumerate}
	\item the category $\ogpos$ has pushouts of inclusions;
	\item if $U$ and $V$ are isomorphic regular molecules, they are isomorphic \emph{in a unique way};
	\item input and output boundaries of regular molecules are regular molecules;
	\item if $U$ and $V$ are round, then a pair of isomorphisms between $\bord{}{\alpha}U$ and $\bord{}{\alpha}V$ for $\alpha \in \{+, -\}$ extends uniquely to an isomorphism between $\bord{}{}U$ and $\bord{}{}V$.
\end{enumerate}
The first three imply that $U \cp{k} V$ is well-defined and does not depend on a choice of isomorphism between $\bord{k}{+}U$ and $\bord{k}{-}V$. 
The fourth implies that $U \celto V$ can be uniquely constructed by extending the isomorphisms $\bord{}{\alpha}U \incliso \bord{}{\alpha}V$ to an isomorphism $\bord U \incliso \bord V$, then gluing $U$ and $V$ along this isomorphism, and finally adding a greatest element with the appropriate orientation.
\end{comm}

\begin{exm}
Let $\data{arrow} \eqdef (\bullet \celto \bullet)$ and $\data{binary} \eqdef ((\data{arrow} \cp{0} \data{arrow}) \celto \data{arrow})$. 
The shape of the diagram of Example \ref{exm:whisker} is generated as $\data{binary} \cp{0} \data{arrow}$,
while the oriented graded poset of Example \ref{exm:frob} is generated as $(\data{cobinary} \cp{0} \data{arrow}) \cp{1} (\data{arrow} \cp{0} \data{binary})$, where $\data{cobinary} \eqdef (\data{arrow} \celto (\data{arrow} \cp{0} \data{arrow}))$.
\end{exm}

\begin{rmk}
As discussed in \cite[\S 2.1]{hadzihasanovic2020diagrammatic}, the pasting constructions $-\cp{k}-$ satisfy the equations of composition in strict $\omega$\nbd categories \emph{up to unique isomorphism}. 
It follows that the ``same'' regular molecule may be constructed in different ways.
For example, letting $\data{globe} \eqdef (\data{arrow} \celto \data{arrow})$, we have
\begin{center}
	$(\data{globe} \cp{0} \data{arrow}) \cp{1} (\data{arrow} \cp{0} \data{globe}) 
	\; \simeq \;
	\data{globe} \cp{0} \data{globe}
	\; \simeq \;
	(\data{arrow} \cp{0} \data{globe}) \cp{1} (\data{globe} \cp{0} \data{arrow}).$

	
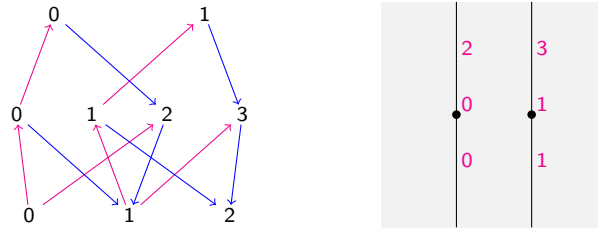
\begin{tikzpicture}[xscale=4, yscale=4, baseline={([yshift=-.5ex]current bounding box.center)}]
\path[fill=white] (0, 0) rectangle (1, 1);
\draw[->, draw=magenta] (0.1625, 0.19999999999999998) -- (0.12916666666666665, 0.4666666666666667);
\draw[->, draw=magenta] (0.2125, 0.19999999999999998) -- (0.5791666666666667, 0.4666666666666667);
\draw[->, draw=magenta] (0.4875, 0.19999999999999998) -- (0.3875, 0.4666666666666667);
\draw[->, draw=magenta] (0.5375, 0.19999999999999998) -- (0.8375, 0.4666666666666667);
\draw[->, draw=blue] (0.16249999999999998, 0.4666666666666667) -- (0.4625, 0.19999999999999998);
\draw[->, draw=magenta] (0.1375, 0.5333333333333333) -- (0.2375, 0.7999999999999999);
\draw[->, draw=blue] (0.42083333333333334, 0.4666666666666667) -- (0.7875, 0.19999999999999998);
\draw[->, draw=magenta] (0.4125, 0.5333333333333333) -- (0.7125, 0.7999999999999999);
\draw[->, draw=blue] (0.6125, 0.4666666666666667) -- (0.5125, 0.19999999999999998);
\draw[->, draw=blue] (0.8708333333333333, 0.4666666666666667) -- (0.8374999999999999, 0.19999999999999998);
\draw[->, draw=blue] (0.2875, 0.7999999999999999) -- (0.5875, 0.5333333333333333);
\draw[->, draw=blue] (0.7625, 0.7999999999999999) -- (0.8625, 0.5333333333333333);
\node[text=black, font={\scriptsize \sffamily}, xshift=0pt, yshift=0pt] at (0.16666666666666666, 0.16666666666666666) {0};
\node[text=black, font={\scriptsize \sffamily}, xshift=0pt, yshift=0pt] at (0.5, 0.16666666666666666) {1};
\node[text=black, font={\scriptsize \sffamily}, xshift=0pt, yshift=0pt] at (0.8333333333333333, 0.16666666666666666) {2};
\node[text=black, font={\scriptsize \sffamily}, xshift=0pt, yshift=0pt] at (0.125, 0.5) {0};
\node[text=black, font={\scriptsize \sffamily}, xshift=0pt, yshift=0pt] at (0.375, 0.5) {1};
\node[text=black, font={\scriptsize \sffamily}, xshift=0pt, yshift=0pt] at (0.625, 0.5) {2};
\node[text=black, font={\scriptsize \sffamily}, xshift=0pt, yshift=0pt] at (0.875, 0.5) {3};
\node[text=black, font={\scriptsize \sffamily}, xshift=0pt, yshift=0pt] at (0.25, 0.8333333333333333) {0};
\node[text=black, font={\scriptsize \sffamily}, xshift=0pt, yshift=0pt] at (0.75, 0.8333333333333333) {1};
\end{tikzpicture}
	\quad \quad \quad
	\begin{tikzpicture}[xscale=3, yscale=3, baseline={([yshift=-.5ex]current bounding box.center)}]
\path[fill=gray!10] (0, 0) rectangle (1, 1);
\draw[color=black, opacity=1] (0.6666666666666666, 0.5) .. controls (0.6666666666666666, 0.5) and (0.6666666666666666, 0.5833333333333333) .. (0.6666666666666666, 0.75);
\draw[color=black, opacity=1] (0.6666666666666666, 1) .. controls (0.6666666666666666, 1.0) and (0.6666666666666666, 0.9166666666666666) .. (0.6666666666666666, 0.75);
\draw[color=black, opacity=1] (0.3333333333333333, 0.5) .. controls (0.3333333333333333, 0.5) and (0.3333333333333333, 0.5833333333333333) .. (0.3333333333333333, 0.75);
\draw[color=black, opacity=1] (0.3333333333333333, 1) .. controls (0.3333333333333333, 1.0) and (0.3333333333333333, 0.9166666666666666) .. (0.3333333333333333, 0.75);
\draw[color=black, opacity=1] (0.6666666666666666, 0.5) .. controls (0.6666666666666666, 0.5) and (0.6666666666666666, 0.41666666666666663) .. (0.6666666666666666, 0.25);
\draw[color=black, opacity=1] (0.6666666666666666, 0) .. controls (0.6666666666666666, 0.0) and (0.6666666666666666, 0.08333333333333333) .. (0.6666666666666666, 0.25);
\draw[color=black, opacity=1] (0.3333333333333333, 0.5) .. controls (0.3333333333333333, 0.5) and (0.3333333333333333, 0.41666666666666663) .. (0.3333333333333333, 0.25);
\draw[color=black, opacity=1] (0.3333333333333333, 0) .. controls (0.3333333333333333, 0.0) and (0.3333333333333333, 0.08333333333333333) .. (0.3333333333333333, 0.25);
\node[circle, fill=black, draw=black, inner sep=1pt] at (0.3333333333333333, 0.5) {};
\node[circle, fill=black, draw=black, inner sep=1pt] at (0.6666666666666666, 0.5) {};
\node[text=magenta, font={\scriptsize \sffamily}, xshift=4pt, yshift=4pt] at (0.6666666666666666, 0.75) {3};
\node[text=magenta, font={\scriptsize \sffamily}, xshift=4pt, yshift=4pt] at (0.3333333333333333, 0.75) {2};
\node[text=magenta, font={\scriptsize \sffamily}, xshift=4pt, yshift=4pt] at (0.6666666666666666, 0.25) {1};
\node[text=magenta, font={\scriptsize \sffamily}, xshift=4pt, yshift=4pt] at (0.3333333333333333, 0.25) {0};
\node[text=magenta, font={\scriptsize \sffamily}, xshift=4pt, yshift=4pt] at (0.3333333333333333, 0.5) {0};
\node[text=magenta, font={\scriptsize \sffamily}, xshift=4pt, yshift=4pt] at (0.6666666666666666, 0.5) {1};
\end{tikzpicture}
\end{center}
\end{rmk}

\begin{impl}
We want to implement regular molecules as a subtype $\data{Shape}$ of $\data{OgPoset}$ with a nullary constructor $\data{point}$ and partial binary constructors $\data{atom}(-, -)$ and $\data{paste}_k(-, -)$ for $k \in \mathbb{N}$. In order to implement the constructors, we need to be able to perform the following operations:
\begin{enumerate}
\item compute input and output $k$\nbd boundaries;
\item check if a closed subset is round;
\item determine if two regular molecules are isomorphic;
\item compute the pushout of a span of inclusions.
\end{enumerate}
The first, second, and fourth of these admit straightforward algorithms of low-degree polynomial time complexity, that do not rely on any special properties of regular molecules.
\begin{codecomm}
The class \code{rewalt.ogposets.Closed} has an attribute \code{isround} which computes and returns whether a closed subset of an oriented graded poset is round. 

The class \code{rewalt.ogposets.OgMapPair} is an overlay for pairs of \code{OgMap}s, with a method \code{pushout} returning the pushout cospan of a span of total \code{OgMap}s, when well-defined.
\end{codecomm}
The third problem, however, is non-trivial. Indeed, the isomorphism problem generalised to all oriented graded posets is equivalent to the \emph{graph isomorphism} ($\data{GI}$) problem, which is not known to be in $\data{P}$; the best known algorithm, due to Babai, runs in quasipolynomial time \cite{babai2016graph}.

\begin{rmk}
As customary in this context, a graph is a \emph{simple} graph (no loops or multiple edges).
\end{rmk}

\begin{prop} \label{prop:gicomplete}
The isomorphism problem for oriented graded posets is $\data{GI}$\nbd complete.
\end{prop}
\begin{proof}
Deciding isomorphism of oriented graded posets is equivalent to deciding isomorphism of their Hasse diagrams with $\{+, -\}$\nbd labelled edges.
The isomorphism problem for edge-labelled finite graphs is an instance of the isomorphism problem for finite relational structures, which is $\data{GI}$\nbd complete \cite{miller1979graph}.

Conversely, a directed graph can be represented by its ``oriented incidence poset'': the 0\nbd dimensional elements are the vertices, the 1\nbd dimensional elements are the edges, the only input face of an edge is its source, and the only output face of an edge is its target.
Two directed graphs are isomorphic if and only if their oriented incidence posets are isomorphic.
Since $\data{GI}$ reduces to the isomorphism problem for directed graphs, it reduces to the isomorphism problem for 1\nbd dimensional oriented graded posets.
\end{proof}

Nevertheless, in the special case of regular molecules, we can do much better. 
Our strategy is to describe a deterministic \emph{traversal algorithm}, where the traversal order depends only on the intrinsic structure of a regular molecule as an oriented graded poset and not on its representation.

Given $U, V: \data{OgPoset}$ representing regular molecules, we traverse both $U$ and $V$, and then \emph{reorder} their elements in each dimension according to their traversal order.
If $U', V': \data{OgPoset}$ are the reordered versions of $U, V$, we then have
\begin{center}
	$U \simeq V$ if and only if $U' \equiv V'$.
\end{center}
We will show that, with this strategy, we can solve the isomorphism problem for regular molecules in time $O(n^3 \log n)$. A more precise upper bound is given in Theorem \ref{thm:complexity} below.

In addition to solving the isomorphism problem for regular molecules, the traversal order gives us a canonical form for regular molecules in $\data{OgPoset}$ form.
If we implement the constructors of $\data{Shape}$ in such a way that they always produce an $\data{OgPoset}$ in traversal order, we obtain that
\begin{center}
	 for all $U, V: \data{Shape}$, $U \simeq V$ if and only if $U \equiv V$,
\end{center}
that is, we have a \emph{unique representation} for shapes of diagrams.

\begin{codecomm}
The class \code{rewalt.shapes.Shape} has constructors \code{point()}, \code{atom(u, v)}, and \code{paste(u, v, [k])}; default is $\code{k} \eqdef \min (\dmn{U}, \dmn{V}) - 1$.
\end{codecomm}

\begin{figure}[t!] 
\begin{algor} \small
\begin{algorithmic}[5]
\Procedure {Traverse} {$U:$ regular molecule}
\State $\data{marked} \gets []$
\State $\data{stack} \gets [U]$
\While {$\data{stack}$ is not empty}
	\State $\data{focus} \gets$ top of $\data{stack}$
	\State $\data{dim} \gets \dmn{\data{focus}}$
	\If {$\data{focus} \subseteq \data{marked}$}
		\State pop $\data{focus}$ from top of $\data{stack}$
	\Else
		\If {$\bord{}{-}\data{focus} \not\subseteq \data{marked}$}
			\State push $\bord{}{-}\data{focus}$ to top of $\data{stack}$
		\Else
			\If {$\data{focus} = \clos\{x\}$ for some $x$}
				\State append $x$ to $\data{marked}$
				\State pop $\data{focus}$ from top of $\data{stack}$
				\If {$\bord{}{+}\data{focus} \not\subseteq \data{marked}$}
					\State push $\bord{}{+}\data{focus}$ on top of $\data{stack}$
				\EndIf
			\Else
				\State $y \gets$ first item of dimension $\data{dim} - 1$ in $\data{marked}$ such that
				\State $\quad \quad y$ has an unmarked input coface in $\data{focus}$
				\State $x \gets$ unique input coface of $y$ in $\data{focus}$
				\State push $\clos\{x\}$ on top of stack
			\EndIf				
		\EndIf
	\EndIf
\EndWhile
\State \Return $\data{marked}$
\EndProcedure
\end{algorithmic}
\end{algor}
\caption{The traversal algorithm.} \label{fig:traversal}
\end{figure}

The algorithm is described in Figure \ref{fig:traversal}. At each iteration of the main loop (line 4), the current state is fully described by the \emph{stack} -- including its top element, the \emph{focus} -- and by the list of \emph{marked} elements.
\end{impl}

\begin{lem} \label{lem:stackmolecule}
Let $V$ be an item on the stack. Then $V$ is a regular molecule. If $W$ is below $V$ on the stack, then $V$ is a proper subset of $W$.
\end{lem}
\begin{proof}
Initially, the stack only contains $U$, which is a regular molecule by assumption.
Assume, inductively, that the statement is true at the beginning of the current iteration with focus $V$, and that a set $V'$ is pushed onto the stack at the end. Then either
\begin{enumerate}
	\item $V' = \bord{}{\alpha}V$ for some $\alpha \in \{+, -\}$, or
	\item $V' = \clos\{x\}$ for some $x \in V$.
\end{enumerate}
In both cases, $V'$ is a regular molecule and a proper subset of $V$ (hence also of each item below $V$), under the assumption that $V$ is a regular molecule.
\end{proof}

\begin{rmk} \label{rmk:stackclassification}
In fact, any $V$ that appears on the stack is either $\bord{k}{-}U$, which we call ``$U$-linked'', or it is $\clos\{x\}$ or $\bord{k}{\alpha}x$, which we call ``$x$-linked'', for some $x \in U$. In the latter case, $V$ is \emph{round}, which implies that it is also \emph{pure} \cite[Lemma 1.35]{hadzihasanovic2020diagrammatic}: its maximal elements all have the same dimension.
\end{rmk}

\begin{lem} \label{lem:fullymarked}
Suppose $V$ is on the stack. Then all elements of $V$ must be marked before any item below $V$ is accessed, or before any proper superset of $V$ becomes the focus.
\end{lem}
\begin{proof}
By Lemma \ref{lem:stackmolecule}, as long as $V$ is on the stack, only $V$ and its proper subsets can be on top.
It follows that, for a proper superset of $V$ to be the focus, $V$ must be popped from the stack at the end of an iteration where $V$ is the focus.
There are only two ways this can happen:
\begin{itemize}
	\item $V$ was already fully marked before the current loop iteration, or
	\item $\bord{}{-}V$ was fully marked and $V = \clos\{x\}$ for some $x$ which is marked at the current loop iteration.
\end{itemize}
In both cases, $\bord{}{-}V$ was already fully marked before the current loop iteration.
In the latter case, if $\bord{}{+}V$ is already fully marked, then $V = \{x\} \cup \bord{}{-}V \cup \bord{}{+}V$ is also fully marked. 
Otherwise, $\bord{}{+}V \subsetneq V$ gets pushed onto the stack to replace $V$, and must be popped before any superset of $V$ becomes the focus.
By the same case distinction, whenever $\bord{}{+}V$ is popped, either
\begin{itemize}
	\item it was fully marked, in which case $V$ was fully marked, or
	\item it is of the form $\clos\{y\}$ for some $y$ which is marked at the current loop iteration.
\end{itemize}
Either way, since all regular molecules satisfy the \emph{globularity} property $\bord{}{\alpha}(\bord{}{+}V) = \bord{}{\alpha}(\bord{}{-}V) \subseteq \bord{}{-}V$, we know that $\bord{}{+}V$, hence $V$, is fully marked at the end of the iteration, and nothing is added to the stack.
\end{proof}

\begin{lem} \label{lem:norepetitions}
Any subset $V$ of $U$ can be pushed onto the stack at most once.
\end{lem}
\begin{proof}
Suppose $V$ is pushed onto the stack. As long as $V$ is on the stack, any subsequent addition to the stack must be a proper subset of $V$, so it cannot be equal to $V$.

If $V$ is popped from the stack, by Lemma \ref{lem:fullymarked}, it must be fully marked before any item below it is accessed.
Since the algorithm checks if a set is fully marked before pushing it onto the stack, $V$ can never appear again.
\end{proof}

\begin{lem} \label{lem:topdimsuffices}
Let $V$ be the focus, $n \eqdef \dmn{V}$. Then either $V$ is fully marked, or there exists an $n$\nbd dimensional element of $V$ which is unmarked.
\end{lem}
\begin{proof}
First, we prove a weaker result: either $V$ is fully marked, or there exists a \emph{maximal} element of $V$ which is unmarked. 

Let $x \in V$ be marked. At some prior iteration, $\clos\{x\}$ must have been the focus, and by Lemma \ref{lem:fullymarked}, in order for $V$ to become the focus, $\clos\{x\}$ must have been fully marked as well. Because
\begin{equation*}
	V = \bigcup_{k \leq n} \clos{\maxd{k}{V}} = \bigcup_{k \leq n} \bigcup_{x \in \maxd{k}{V}} \clos{\{x\}},
\end{equation*}
it follows that $V$ is fully marked if and only if its maximal elements are all marked.

Now, $V$ has one of the two forms in Remark \ref{rmk:stackclassification}. If $V$ is of the second form, its maximal elements all have the top dimension, so we only need to consider the case $V = \bord{k}{-}U$.

At the start of the algorithm, $U, \ldots, \bord{0}{-}U$ are all consecutively added to the stack.
So $\bord{k}{-}U$ becomes the focus either at this stage, in which case \emph{all} its elements are unmarked, or after $\bord{k-1}{-}U$ is fully marked. 
In the latter case, any maximal element of $\bord{k}{-}U$ of dimension strictly smaller than $k$ also belongs to $\bord{k-1}{-}U$.
\end{proof}

\begin{thm} \label{thm:traversalcorrectness}
The traversal algorithm is correct: given a regular molecule $U$, it terminates returning a unique linear ordering of the elements of $U$. 
\end{thm}
\begin{proof}
As a particular case of Lemma \ref{lem:fullymarked}, $U$ must be fully marked before the stack is emptied. Therefore, the algorithm either terminates after all elements have been traversed, or it does not terminate.

To prove that the algorithm does always terminate, it suffices to show that, unless all elements are already marked, it always finds an element to mark.
First of all, observe that, from any state, the algorithm first goes through the following sequence of steps:
\begin{enumerate}
	\item popping all fully marked subsets from the top of the stack;
	\item once it reaches a subset which is not fully marked, successively pushing its lower\nbd dimensional input boundaries that are not fully marked onto the stack.
\end{enumerate}
At the end of this sequence, we always reach a state in which the focus $V$ is not fully marked, but $\bord{}{-} V$ is fully marked. Let us call such a $V$ a \emph{proper} focus.

We proceed by induction on dimension and proper subsets of a proper focus. If $\dmn{V} = 0$, since a 0\nbd molecule always consists of a single element, $V = \{x\}$, and $x$ gets marked at the current iteration.

Let $n \eqdef \dmn{V}$. By Lemma \ref{lem:topdimsuffices}, there is an unmarked $x \in V_n$. If $V = \clos\{x\}$, then $x$ is marked at the current iteration, and we are done. Otherwise, we prove that there always exists a pair $(y, x)$ where $x \in V_n$ is unmarked, and $y$ is a marked input face of $x$. 
By \cite[Lemma 1.16]{hadzihasanovic2020diagrammatic} applied to $V$, the coface $x$ is unique given $y$, so among such pairs we can pick the one where $y$ comes \emph{earliest} in the list of marked elements, and this selects a unique $x$.

Let $x \in V_n$ be unmarked. By a dual version of [\textit{ibid.}, Lemma 1.37], there exists a sequence 
\begin{equation*}
	y_0 \to x_0 \to \ldots \to y_m \to x_m = x
\end{equation*}
where $y_0 \in \sbord{n-1}{-}V$, $x_i \in V_n$, $y_i$ is an input face of $x_i$, and $y_{i+1}$ is an output face of $x_i$. 
Since $V$ is a proper focus, $y_0$ is marked. Let $k$ be the smallest index such that $x_k$ is unmarked; because $x_m$ is unmarked, such a $k$ exists. 
Then $x_i$ is marked for all $i < k$, hence $\clos\{x_i\}$ is also marked.
It follows that $y_k \in \bord{}{+}x_{k-1}$ is marked, and the pair $(y_k, x_k)$ satisfies our requirement.

Thus, the algorithm will find a unique $x \in V_n$ and push $\clos\{x\}$ onto the stack.
The next proper focus will necessarily be a proper subset of $V$, and we conclude by the inductive hypothesis.
\end{proof}

\begin{dfn}
In what follows, for a fixed regular molecule $U$, we let $|E_n|$ be the number of edges between $n$ and $(n-1)$\nbd dimensional elements in the Hasse diagram of $U$, and we let
\begin{equation*}
	|U_\mathrm{max}| \eqdef \max_{n} |U_n|, \quad \quad |E_\mathrm{max}| \eqdef \max_{n} |E_n|.
\end{equation*}
\end{dfn}

\begin{thm} \label{thm:complexity}
The traversal algorithm admits an implementation running in time
\begin{equation*} 
O\big(|U|^2(|E_\mathrm{max}|\cdot \log |E_\mathrm{max}| + |U_\mathrm{max}|\cdot \log |U_\mathrm{max}|)\big).
\end{equation*}
\end{thm}
\begin{proof}
First of all, we represent any closed set on the stack with its graded set of maximal elements. 
To initialise the algorithm, we only need to compute the maximal elements of $U$. 
This can be done in time $O(|U|)$ by going through the elements of $U$ and checking if their set of cofaces is empty.

Next, let us find an upper bound for the number of iterations of the main loop (line 4).
Let $V$ be a set on the stack, $n \eqdef \dmn{V}$. Then $V$ can become the focus
\begin{itemize}
	\item at most once before pushing $\bord{}{-} V$ onto the stack (line 11),
	\item at most once before pushing $\clos\{y\}$ onto the stack for each $y \in V_n$ (line 22), and
	\item at most once to be popped from the stack (line 8),
\end{itemize}
after which, by Lemma \ref{lem:norepetitions}, it can never appear again. 
Thus, the number of loop iterations with $V$ as focus is bounded by $|V_n| + 2$. 

By Remark \ref{rmk:stackclassification}, every set $V$ on the stack is either ``$U$\nbd linked'' or ``$x$\nbd linked'' for some $x \in U$. 
There are $(\dmn{U} + 1)$ many $U$\nbd linked focusses and $(2\dmn{x} + 1)$ many $x$\nbd linked focusses.
Then
\begin{itemize}
	\item the number of loop iterations with $U$\nbd linked focusses is bounded by $|U| + 2\dmn{U} + 2$, and
	\item for each $x$, the number of iterations with $x$\nbd linked focusses is bounded by $|\clos\{x\}| + 4\dmn{x} + 2$.
\end{itemize}
Since there are $|U|$ elements, $|\clos\{x\}| \leq |U|$, and $\dmn{x} \leq \dmn{U}$, we have a coarse upper bound of $(|U| + 1)(|U| + 4\dmn{U} + 2)$ on the total number of iterations, which is $O(|U|^2)$.

Next, in our implementation, we split the list of marked elements into three objects: a list $\data{order}$ (for the total traversal order), an array of lists $\data{grorder}$ (for the traversal order split by dimension), and a graded set $\data{marked}$ (for the set of marked elements). 

Consider a single loop iteration with focus $V$, $n \eqdef \dmn{V}$.

\noindent \textbf{(Line 7).} 
By Lemma \ref{lem:topdimsuffices}, to check if $V$ is fully marked, it suffices to check whether $V_n \subseteq \data{marked}_n$.
Since both are sorted arrays of integers, they can be compared in time linear in $|V_n| + |\data{marked}_n|$, which is $O(|U_n|)$.
At this stage, we may also record the unmarked $n$\nbd dimensional elements of $V$ in a sorted array $\data{unmarked}$ without affecting the complexity.

\noindent \textbf{(Line 10).} To compute the maximal elements of $\bord{}{-}V$ and $\bord{}{+}V$, we may use different strategies depending on whether $V$ is ``$U$\nbd linked'' or not.

If $V = \bord{n}{-}U$, we compute the $(n-1)$\nbd dimensional elements of $\bord{}{-}V = \bord{n-1}{-}U$ simply by going through the elements of $U_{n-1}$ and checking which ones have empty sets of output cofaces, in time $O(|U_{n-1}|)$.
Lower\nbd dimensional maximal elements are shared between $V$ and $\bord{}{-}V$, so we may then point from the latter to the former, at no extra cost.

If $V$ is not $U$\nbd linked, $V$ and its boundaries are pure, so the set of maximal elements of $\bord{}{\alpha}V$ is equal to $\sbord{}{\alpha} V$, and each of its elements is covered by an element of $V_n$.
To compute it, we add all the input and output faces of all $x \in V_n$ to sets $\data{in\_faces}$ and $\data{out\_faces}$, respectively, then use the relations $\sbord{}{-}V = \data{in\_faces} \setminus \data{out\_faces}$ and $\sbord{}{+}V = \data{out\_faces} \setminus \data{in\_faces}$.

There are $O(|E_n|)$ faces of elements of $V_n$, and we can sort $\data{in\_faces}$ and $\data{out\_faces}$, remove duplicates, and compute their difference in time $O(|E_n|\cdot \log |E_n|)$.

At this stage, we also create an associative array $\data{candidates}$ as follows: whenever $x \in V_n$ is in $\data{unmarked}$, and $y$ is an input face of $x$, we add the position of $x$ as a value to $\data{candidates}$, indexed by the position of $y$.
We then sort the indices of $\data{candidates}$. 
This also takes time $O(|E_n| \cdot \log |E_n|)$ so it does not affect the overall complexity.

\noindent\textbf{(Lines 10, 16).} By the same reasoning applied to line 7, checking if $\bord{}{-}V$ and $\bord{}{+}V$ are fully marked takes time $O(|U_{n-1}|)$.

\noindent\textbf{(Line 14).} If $V_n$ has a single element that we mark, adding it to $\data{order}$ and $\data{grorder}$ takes constant time with an appropriate implementation of lists.
Adding it to $\data{marked}$ takes $O(|U_n|)$.

\noindent\textbf{(Lines 19---21).} To select the next focus we traverse $\data{grorder}_{n-1}$ starting from the first item and search for each item in the indices of $\data{candidates}$ until we find a hit $y$.
This takes time $O(|U_{n-1}|\cdot \log |U_{n-1}|)$ in the worst case.
The next focus will be $\clos\{x\}$, where $x$ is the value corresponding to index $y$.

Overall, the worst-case complexity is $O(|U_n| + |E_n|\cdot \log |E_n| + |U_{n-1}|\cdot \log |U_{n-1}|)$. Using the bounds $|U_n|, |U_{n-1}| \leq |U_\mathrm{max}|$ and $|E_n| \leq |E_\mathrm{max}|$, and multiplying by our bound on the number of iterations, we conclude. 
\end{proof}

\section{A type theory for higher-dimensional rewriting} \label{sec:diagrams}

\begin{dfn}
We rapidly go through the definitions of diagrammatic sets and some related notions.
For a thorough treatment, we refer to \cite[Section 4 and onwards]{hadzihasanovic2020diagrammatic}, and to \cite[Section V]{hadzihasanovic2021smash} for diagrammatic complexes as presentations of higher\nbd dimensional theories.
\end{dfn}

\begin{dfn}[Diagrammatic set]
Let $\atom$ (to be read \emph{atom}) be a skeleton of the full subcategory of $\ogpos$ on the atoms of every dimension. 
A \emph{diagrammatic set} is a presheaf on $\atom$. Diagrammatic sets and their morphisms of presheaves form a category $\dgmset$.
\end{dfn}

\begin{dfn} 
We identify $\atom$ with a full subcategory $\atom \incl \dgmset$ via the Yoneda embedding. With this identification, we use morphisms in $\dgmset$ as our notation for both elements and structural operations of a diagrammatic set $X$:
\begin{itemize}
	\item $x \in X(U)$ becomes $x\colon U \to X$, and
	\item for each map $f\colon V \to U$ in $\atom$, $X(f)(x) \in X(V)$ becomes $f;x\colon V \to X$.
\end{itemize}
The embedding $\atom \incl \dgmset$ extends along pushouts of inclusions to the full subcategory of $\ogpos$ on the regular molecules.
\end{dfn}

\begin{dfn}[Diagrams and cells] 
Let $X$ be a diagrammatic set and $U$ a regular molecule. A \emph{diagram of shape $U$ in $X$} is a morphism $x\colon U \to X$. A diagram is a \emph{cell} if $U$ is an atom. For all $n \in \mathbb{N}$, we say that $x$ is an \emph{$n$\nbd diagram} or an \emph{$n$\nbd cell} when $\dmn{U} = n$.

If $U$ decomposes as $U_1 \cp{k} U_2$, we write $x = x_1 \cp{k} x_2$ for $x_i \eqdef \imath_i;x$, where $\imath_i$ is the inclusion $U_i \incl U$ for $i \in \{1,2\}$. 
Let $\imath_k^\alpha\colon \bord{k}{\alpha}U \incl U$ be the inclusions of the $k$\nbd boundaries of $U$. 
The \emph{input $k$\nbd boundary} of $x$ is the diagram $\bord{k}{-}x \eqdef \imath_k^-;x$ and the \emph{output $k$\nbd boundary} of $x$ is the diagram $\bord{k}{+}x \eqdef \imath_k^+;x$. 
We write $x\colon y^- \celto y^+$ to express that $\bord{k}{\alpha}x = y^\alpha$ for each $\alpha \in \{+,-\}$.
\end{dfn}

\begin{dfn}[Diagrammatic complex]
For each $n \in \mathbb{N}$, let $\atom_n$ be the full subcategory of $\atom$ on the atoms of dimension $\leq n$, and let $\atom_{-1}$ be the empty subcategory. 
The restriction functor $\dgmset \to \psh{}{\atom_n}$ has a left adjoint; let $\skel{n}{}$ be the comonad induced by this adjunction. 
The \emph{$n$\nbd skeleton} of a diagrammatic set $X$ is the counit $\skel{n}{X} \to X$. 
For all $k \leq n$, the $k$\nbd skeleton factors uniquely through the $n$\nbd skeleton of $X$.

A \emph{diagrammatic complex} is a diagrammatic set $X$ together with a set $\gen{X} = \sum_{n\in \mathbb{N}} \gen{X}_n$ of \emph{generating} cells such that, for all $n \in \mathbb{N}$,
\begin{equation*}
\begin{tikzpicture}[baseline={([yshift=-.5ex]current bounding box.center)}]
	\node (0) at (-1.5,1.5) {$\bigsqcup_{x \in \gen{X}_n} \bord U(x)$};
	\node (1) at (2.5,0) {$\skel{n}{X}$};
	\node (2) at (-1.5,0) {$\skel{n-1}{X}$};
	\node (3) at (2.5,1.5) {$\bigsqcup_{x \in \gen{X}_n} U(x)$};
	\draw[1cinc] (0) to (3);
	\draw[1c] (0) to node[auto,arlabel] {$(\bord x)_{x \in \gen{X}_n}$} (2);
	\draw[1cinc] (2) to (1);
	\draw[1c] (3) to node[auto,arlabel] {$(x)_{x \in \gen{X}_n}$} (1);
	\draw[edge] (1.6,0.2) to (1.6,0.7) to (2.3,0.7);
\end{tikzpicture}
\end{equation*}
is a pushout in $\dgmset$, where $U(x)$ denotes the shape of $x$. A diagrammatic complex is \emph{finite} if $\gen{X}$ is finite.
\end{dfn}

\begin{dfn}[Support-based diagrammatic complex]
Each cell in a diagrammatic complex $(X, \gen{X})$ is uniquely of the form $(p\colon U \surj V,\, x\colon V \to X)$, where $p$ is a surjective map of atoms and $x \in \gen{X}$. 
We let $\data{supp}(p, x) \eqdef x$, the \emph{support} of $(p, x)$.

A \emph{support\nbd based diagrammatic complex} is the quotient of a diagrammatic complex by the relations
\begin{equation} \label{eq:supportbased}
	x \sim y \text{ if and only if } \data{supp}(\imath; x) = \data{supp}(\imath; y) \text{ for all inclusions of atoms } \imath\colon V \incl U,
\end{equation}
for all atoms $U$ and cells $x, y\colon U \to X$.
We let $\dgmcpx$ denote the category of finite, support\nbd based diagrammatic complexes with morphisms of their underlying diagrammatic sets.
\end{dfn}

\begin{dfn}
We define a dependent type theory for diagrammatic sets -- more precisely, for finite, support\nbd based diagrammatic complexes -- that relies on an underlying unique representation of regular molecules and their maps, treated as a ``black box''.
Of course, in the previous section we have provided such an implementation and proved that it is computationally feasible.
Nevertheless, it is useful to separate its abstract properties from the implementation details.
\end{dfn}

\begin{dfn}[$\diagset$]
Let $\mathbb{V}$ be an infinite set of variables. 
We define a type theory $\diagset$ as follows. 

\noindent \textbf{Terms.} A term $t$ is a pair of a regular molecule $U$, the \emph{shape} of $t$, and a function $t\colon U \to \mathbb{V}$.
We write $\slice{t}{U}$ to express that $t$ is a term of shape $U$.
Maps $p\colon U \to V$ act on terms by precomposition: if $\slice{t}{V}$ is a term, then $p^*t \eqdef \slice{(p;t)}{U}$.
In particular, we let $\bord{k}{\alpha}t \eqdef \slice{(\imath_k^\alpha;t)}{\bord{k}{\alpha}V}$ for all $k \in \mathbb{N}$ and $\alpha \in \{+,-\}$.

\noindent \textbf{Types.} A type $A$ is either $\varnothing$ or an expression $t \celto s$ where $t, s$ are terms. We may annotate a term $t$ of shape $U$ with the type $A \eqdef \varnothing$ if $U \equiv \bullet$, and $A \eqdef \bord{}{-}t \celto \bord{}{+}t$ otherwise.

\noindent \textbf{Contexts.} A context $\Gamma$ is a list $x_1: A_1, \ldots, x_n: A_n$ of typed variables. We consider two contexts to be equal if they are equal up to a permutation. If $x: A$ is a typed variable, we say that $x$ has \emph{shape} $\bullet$ if $A \equiv \varnothing$, and $U \celto V$ if $A \equiv \slice{t}{U} \celto \slice{s}{V}$.
We write $\slice{x}{U}: A$ to express that $x: A$ has shape $U$.

\noindent \textbf{Substitutions.} A substitution $\sigma$ is a list $x_1 \mapsto t_1, \ldots, x_n \mapsto t_n$ of assignments of terms to variables. 
We consider two substitutions to be equal if they are equal up to a permutation.

\noindent \textbf{Judgments.} We consider three kinds of judgments:
\begin{itemize}
	\item $\Gamma \vdash \quad$ meaning that $\Gamma$ is a well-formed context,
	\item $\Gamma \vdash t \quad$ meaning that $t$ is a well-formed term in context $\Gamma$, and
	\item $\Delta \vdash \sigma: \Gamma \quad$ meaning that $\sigma$ is a well-formed substitution from context $\Delta$ to context $\Gamma$.
\end{itemize}
The inference rules of $\diagset$ are the following. We use $\langle \rangle$ to indicate the empty list.

\begin{algor}\small
\textbf{Rules for contexts.}
\begin{equation*}
\begin{prooftree}
	{ \quad }
	\justifies
	{ \langle \rangle \vdash }
	\using
	{\scriptstyle \data{init}}
\end{prooftree}
\quad \quad
\begin{prooftree}
	{ \Gamma \vdash}
	\justifies
	{ \Gamma,\, x:\varnothing \vdash }
	\using
	{\scriptstyle \data{pt}}
\quad \quad
\end{prooftree}
\begin{prooftree}
	{ \Gamma \vdash \slice{t}{U}: r^- \celto r^+ \quad \Gamma \vdash \slice{s}{V}: r^- \celto r^+ \quad \text{$U, V$ $\data{round}$} }
	\justifies
	{ \Gamma, \,x: t \celto s \vdash}
	\using
	{\scriptstyle \data{gen}}
\end{prooftree}
\end{equation*}
\hfill (where $x \in \mathbb{V}$ is fresh)

\textbf{Rules for terms.}
\begin{equation*}
\begin{prooftree}
	{ \Gamma \vdash \qquad (\slice{x}{V} : A) \in \Gamma \qquad \text{$U$ $\data{atom}$} \qquad \text{$p\colon U \surj V$ $\data{surjective}$} }
	\justifies
	{ \Gamma \vdash \slice{p^* \widehat{x}}{U} }
	\using
	{\scriptstyle\data{cell}}
\end{prooftree}
\end{equation*}

\begin{equation*}
\begin{prooftree}
	{ \Gamma \vdash \slice{t}{U} \qquad \Gamma \vdash \slice{s}{V} \qquad \bord{k}{+}t \equiv \bord{k}{-}s }
	\justifies
	{ \Gamma \vdash \slice{(t \cp{k} s)}{(U \cp{k} V)} }
	\using
	{\scriptstyle\data{paste}_k}, \quad {k < \min (\dmn{U}, \dmn{V})}
\end{prooftree}
\end{equation*}

\textbf{Rules for substitutions.}
\begin{equation*}
\begin{prooftree}
	{ \Gamma \vdash }
	\justifies
	{ \Gamma \vdash \langle \rangle: \Gamma}
	\using
	{\scriptstyle\data{id}}	
\end{prooftree}
\quad \quad
\begin{prooftree}
	{ \Delta \vdash \sigma: \Gamma \qquad \Gamma,\, x: \slice{s}{U} \celto \slice{r}{V} \vdash \qquad \Delta \vdash \slice{t}{U \celto V}: s[\sigma] \celto r[\sigma] }
	\justifies
	{ \Delta \vdash \langle\sigma,\, x \mapsto t\rangle : (\Gamma,\, x: s \celto r) }
	\using
	{\scriptstyle\data{ext}}
\end{prooftree}
\end{equation*}
\end{algor}
In the rules $\data{cell}$ and $\data{paste}$, the terms $\widehat{x}$ and $t \cp{k} s$ are defined as follows:
\begin{itemize}
	\item $\widehat{x}$ is the unique term of shape $V$ which sends the greatest element of $V$ to $x$, and, if $A \equiv t \celto s$, is equal to $t$ on $\bord{}{-}V$ and to $s$ on $\bord{}{+}V$;
	\item $t \cp{k} s$ is the unique term of shape $U \cp{k} V$ that is equal to $t$ on $U \incl (U \cp{k} V)$ and to $s$ on $V \incl (U \cp{k} V)$. 
\end{itemize}
The side conditions for $\data{gen}$ and $\data{paste}$ ensure that this is well-defined. 

To define the action $t[\sigma]$ of a well-formed substitution $\sigma$ on a term $t$, we extend $\sigma$ to a function $\mathbb{V} \to \mathbb{V}$ as follows: for all $x \in \mathbb{V}$, if $(x \mapsto \slice{t}{U}) \in \sigma$, we let $\sigma(x) \eqdef t(\top)$, where $\top$ is the greatest element of $U$; otherwise, $\sigma(x) \eqdef x$.
Then $t[\sigma]$ is the composite of $t\colon U \to \mathbb{V}$ and $\sigma\colon \mathbb{V} \to \mathbb{V}$.
Note that this is well-defined because a well-formed substitution assigns to each variable a term whose shape is an atom.
\end{dfn}

\begin{dfn}[Syntactic category]
The syntactic category $\syncat$ has
\begin{itemize}
	\item well\nbd formed contexts $\Gamma$ as objects, and
	\item well\nbd formed substitutions as morphisms from $\Delta$ to $\Gamma$,
\end{itemize}
with the obvious composition of substitutions, and empty substitutions as identities.
\end{dfn}

\begin{thm} \label{thm:synchar}
The category $\opp{\syncat}$ is equivalent to $\dgmcpx$.
\end{thm}
\begin{proof}[Sketch of proof]
We define an encoding $\data{enc}$ of finite support\nbd based diagrammatic complexes, diagrams, and morphisms as contexts, terms, and substitutions. Given $(X, \gen{X})$, we pick an injective function $\data{name}\colon \gen{X} \to \mathbb{V}$, assigning unique variable names to the generating cells of $X$.

For all diagrams $d\colon U \to X$, we define a term $\data{enc}(d)$ as follows: for all $x \in U$, we let $\data{enc}(d)(x)$ be equal to $\data{name}(\data{supp}(\restr{d}{\clos\{x\}}))$.
Since $(X, \gen{X})$ is support-based, $\data{enc}(d) \equiv \data{enc}(d')$ implies $d = d'$.

Let $n$ be the greatest dimension in which $\gen{X}_n$ is non-empty, and pick a linear ordering $x_1, \ldots, x_{m_k}$ of $\gen{X}_k$ for all $k \leq n$. 
We let $\data{enc}(X, \gen{X}) \eqdef \Gamma_0, \ldots, \Gamma_n$, where
\begin{equation*}
	\Gamma_k \eqdef \data{name}(x_1): \data{enc}(\bord{}{-}x_1) \celto \data{enc}(\bord{}{+}x_1), \,\ldots, \,\data{name}(x_{m_k}): \data{enc}(\bord{}{-}x_{m_k}) \celto \data{enc}(\bord{}{+}x_{m_k}).
\end{equation*}
By the construction of $X$ as a colimit of its generating cells, any map $X \to Y$ is uniquely determined by what it does on $\gen{X}$.
Given a map $f\colon (X, \gen{X}) \to (Y, \gen{Y})$ in $\dgmcpx$, we let $\data{enc}(f)$ be the substitution
\begin{equation*}
	\langle\data{name}_X(x) \mapsto \data{enc}_Y(f(x))\rangle_{x \in \gen{X}}.
\end{equation*}
Conversely, we define an interpretation $\intp{-}$ of well\nbd formed contexts, terms, and substitutions by induction on inference rules of $\diagset$. 
At each step the interpretation $\intp{\Gamma}$ of a well\nbd formed context is a support-based diagrammatic complex with one generator $\intp{\widehat{x}}$ of shape $U$ for each variable $\slice{x}{U}$ in $\Gamma$.
\begin{itemize}
\item $(\data{init})$ The interpretation of the empty context is the initial diagrammatic set.
\item $(\data{pt})$ Suppose $\intp{\Gamma}$ is defined. The interpretation of $\Gamma,\, x:\varnothing$ is the coproduct $\intp{\Gamma} + \bullet$.
The interpretation of $\widehat{x}$ is the inclusion $\bullet \incl \intp{\Gamma} + \bullet$.
\item $(\data{gen})$ Suppose $\intp{\Gamma}$ and $\intp{\slice{t}{U}}, \intp{\slice{s}{V}}$ are defined.
	The interpretation of $\Gamma, \,x: t \celto s$ is the pushout of $\bord \intp{\widehat{x}}\colon \bord (U \celto V) \to \intp{\Gamma}$ and $\bord (U \celto V) \incl (U \celto V)$, quotiented by the equations (\ref{eq:supportbased}), where $\bord \intp{\widehat{x}}$ is equal to $\intp{t}$ on $\bord{}{-}(U \celto V)$ and to $\intp{s}$ on $\bord{}{+}(U \celto V)$. 
\item $(\data{cell})$ Suppose $\intp{\Gamma}$ is defined and has a generating cell $\intp{\widehat{x}}$.
The interpretation of $p^*\widehat{x}$ is $p;\intp{\widehat{x}}$.
\item $(\data{paste}_k)$ Suppose $\intp{\Gamma}$ and $\intp{t}, \intp{s}$ are defined with $\bord{k}{+}\intp{t} = \intp{\bord{k}{+}t} = \intp{\bord{k}{-}s} = \bord{k}{-}\intp{s}$.
The interpretation of $t \cp{k} s$ is the diagram $\intp{t} \cp{k} \intp{s}$.
\item $(\data{id})$ The interpretation of the empty substitution in context $\Gamma$ is the identity of $\intp{\Gamma}$.
\item $(\data{ext})$ Suppose $\intp{\sigma}$ and $\intp{t}$ are defined, where $\intp{\widehat{x}}$ and $\intp{t}$ both have the same shape $U$.
By the construction of $\intp{\Gamma, x}$ as a colimit of $\intp{\Gamma}$ and $U$, the pair of $\intp{\sigma}\colon \intp{\Gamma} \to \intp{\Delta}$ and $\intp{t}\colon U \to \intp{\Gamma}$ induces a unique morphism $\intp{\sigma, x \mapsto t}\colon \intp{\Gamma, x} \to \intp{\Delta}$.
\end{itemize}
It is routine to check that $\data{enc}$ and $\intp{-}$ define contravariant functors between $\dgmcpx$ and $\syncat$, and that they are each other's inverse up to natural isomorphism.
\end{proof}

\begin{rmk}
The proof of Theorem \ref{thm:synchar} gives a semantic characterisation of well\nbd formed terms as diagrams in a diagrammatic set.
An immediate consequence is that the following rule is admissible:
{\small \begin{equation*}
\begin{prooftree}
	{ \Gamma \vdash \slice{t}{V} \quad \quad \text{$p\colon U \to V$ $\data{map}$} }
	\justifies
	{ \Gamma \vdash \slice{p^* t}{U} }
	\using
	{\scriptstyle \data{pb}}
\end{prooftree}
\end{equation*} }
where $p$ is an arbitrary map of regular molecules.
\end{rmk}

\begin{comm}
A sticking point in our type theory is the fact that $\data{cell}$ is parametrised by an arbitrary surjective map of atoms $p$.
This is necessary to access the ``weak units'' and degenerate cells which in our framework are needed, among other things, to model nullary operations in an algebraic theory. 

In practice, however, this is the one point in which the underlying implementation of regular molecules and their maps has to be explicitly accessed in order to define $p$ and its domain.
To avoid this, in a practical implementation, we want to include explicitly some extra admissible rules, corresponding to the application of useful maps that are parametric in their codomain.

In particular, we want to explicitly include
\begin{itemize}
\item the trivial case $p \equiv \idd{U}$:
\begin{center}
{ \small
\begin{prooftree}
	{ \Gamma \vdash \qquad (\slice{x}{U} : A) \in \Gamma }
	\justifies
	{ \Gamma \vdash \slice{\widehat{x}}{U} }
	\using
	{\scriptstyle\data{cell'}}
\end{prooftree}
},
\end{center}

\item \emph{unit} rules, modelling \cite[\S 4.16]{hadzihasanovic2020diagrammatic}:
\begin{center}
{ \small
\begin{prooftree}
	{ \Gamma \vdash \slice{t}{U} }
	\justifies
	{ \Gamma \vdash \data{unit}(t) \eqdef \tau^*(t) : t \celto t }
	\using
	{\scriptstyle\data{unit}}
\end{prooftree}
},
\end{center}

\item \emph{left and right unitor rules}, modelling [\textit{ibid.}, \S 4.17]:
\begin{center}
{ \small
\begin{prooftree}
	{ \Gamma \vdash \slice{t}{U} \qquad \text{$V \submol \bord{}{-}U$ $\data{round}$} }
	\justifies
	{ \Gamma \vdash \data{lunitor}_V(t) \eqdef (\lmap{V}{U}{-})^* t }
	\using
	{\scriptstyle\data{lunitor}}
\end{prooftree}
	\quad \quad \quad 
\begin{prooftree}
	{ \Gamma \vdash \slice{t}{U} \qquad \text{$V \submol \bord{}{+}U$ $\data{round}$} }
	\justifies
	{ \Gamma \vdash \data{runitor}_V(t) \eqdef (\rmap{V}{U}{-})^* t }
	\using
	{\scriptstyle\data{runitor}}
\end{prooftree}
}
\end{center}
where $V$ can be specified, for example, by the set of positions of its maximal elements. 
\end{itemize}
We may also have extra rules for \emph{simplex and cube degeneracy} maps and for \emph{cube connection} maps, in the case where $U$ is an oriented simplex or cube as in [\textit{ibid.}, \S 3.33]. All of these are implemented as diagram methods in $\texttt{rewalt}$.
\end{comm}

\begin{exm} \label{exm:lunital}
As an example, we give a presentation in $\diagset$ of the theory of a left\nbd unital binary operation, together with its implementation in $\texttt{rewalt}$. 
In the framework of diagrammatic sets, a many-sorted ``monoidal theory'' is presented by a diagrammatic complex with a single 0\nbd cell; this is analogous to the way a monoidal category is a bicategory with a single 0\nbd cell.
The sorts are generating 1\nbd cells, the basic operations are generating 2\nbd cells, and ``oriented equations'' are generating 3\nbd cells.

First, we add a single 0\nbd cell $x$ and a single sort $a$.

\begin{center}
\noindent \begin{minipage}[c]{0.5\textwidth}
\small
\begin{prooftree}
\[ 
	\[
		{ \quad }
		\justifies
		{ \langle \rangle \vdash }
		\using
	{\scriptstyle\data{init}} 
	\]
	\justifies
	{ x: \varnothing }
	\using
	{\scriptstyle\data{pt}} 
\]
\justifies
{ x: \varnothing \vdash \widehat{x} }
\using
{\scriptstyle\data{cell'}}
\end{prooftree}
$\quad \quad$
\begin{prooftree}
\[
	{ x: \varnothing \vdash \widehat{x} \quad \quad x: \varnothing \vdash \widehat{x}}
	\justifies
	{ x: \varnothing, \, a: \widehat{x} \celto \widehat{x} \vdash }
	\using
	{\scriptstyle\data{gen}}
\]
\justifies
{ x: \varnothing,\, a: \widehat{x} \celto \widehat{x} \vdash \widehat{a} }
\using
{\scriptstyle\data{cell'}}
\end{prooftree}
\end{minipage}%
\begin{minipage}{0.5\textwidth}
\begin{lstlisting}[language=Python]
import rewalt
Lun = rewalt.DiagSet()
x = Lun.add('x')
a = Lun.add('a', x, x)
\end{lstlisting}
\end{minipage}
\end{center}

\noindent Let $\Gamma \eqdef x: \varnothing,\, a: \widehat{x} \celto \widehat{x}$. We add a binary operation $m$.

\begin{center}
\noindent
\begin{minipage}[c]{0.5\textwidth}
\small
\begin{prooftree}
\[
	{ \[
		{ \Gamma \vdash \widehat{a} \quad \quad \Gamma \vdash \widehat{a} }
		\justifies
		{ \Gamma \vdash \widehat{a} \cp{0} \widehat{a} }
		\using
		{\scriptstyle	\data{paste}_0}
	\] \quad \quad \Gamma \vdash \widehat{a} }
	\justifies
	{ \Gamma, \, m: \widehat{a} \cp{0} \widehat{a} \celto \widehat{a} \vdash }
	\using
	{\scriptstyle\data{gen}} 
\]
\justifies 
{\Gamma, \, m: \widehat{a} \cp{0} \widehat{a} \celto \widehat{a} \vdash \widehat{m} }
\using
{\scriptstyle\data{cell'}}
\end{prooftree}
\end{minipage}%
\begin{minipage}{0.5\textwidth}
\begin{lstlisting}[language=Python, firstnumber=5]
m = Lun.add('m', a.paste(a), a)
\end{lstlisting}
\end{minipage}
\end{center}

\noindent Let $\Gamma' \eqdef \Gamma, \, m: \widehat{a} \cp{0} \widehat{a} \celto \widehat{a}$. We produce a weak unit on $x$ and add a nullary operation $u$.

\begin{center}
\noindent \begin{minipage}[c]{0.5\textwidth}
\small
\begin{prooftree}
\[
	{ \[
		{ \Gamma' \vdash \widehat{x} }
		\justifies
		{ \Gamma' \vdash \data{unit}(\widehat{x}) }
		\using
		{\scriptstyle\data{unit}}
	\] \quad \quad \Gamma' \vdash \widehat{a} }
	\justifies
	{ \Gamma', \, u: \data{unit}(\widehat{x}) \celto \widehat{a} \vdash }
	\using
	{\scriptstyle\data{gen}}
\]
\justifies 
{\Gamma', \, u: \data{unit}(\widehat{x}) \celto \widehat{a} \vdash \widehat{u} }
\using
{\scriptstyle\data{cell'}}
\end{prooftree}
\end{minipage}%
\begin{minipage}{0.5\textwidth}
\begin{lstlisting}[language=Python, firstnumber=6]
u = Lun.add('u', x.unit(), a)
\end{lstlisting}
\end{minipage}
\end{center}

\noindent Let $\Gamma'' \eqdef \Gamma', \, u: \data{unit}(\widehat{x}) \celto \widehat{a}$. 
We produce a left unitor 2\nbd cell on $a$, and add an ``oriented equation'' exhibiting the fact that $u$ is a left unit for $m$.

\begin{center} 
\small
\begin{prooftree}
{
\[
	{ 
	\[ 
		{ \Gamma'' \vdash \widehat{u} \quad \quad \Gamma'' \vdash \widehat{a} }
		\justifies
		{ \Gamma'' \vdash \widehat{u} \cp{0} \widehat{a} }
		\using
		{\scriptstyle\data{paste}_0}
	\]
	\quad \quad 
	\Gamma'' \vdash \widehat{m} }
	\justifies
	{\Gamma'' \vdash (\widehat{u} \cp{0} \widehat{a}) \cp{1} \widehat{m} }
	\using
	{\scriptstyle\data{paste}_1}
\]
\quad \quad
\[
	{ \Gamma'' \vdash \slice{\widehat{a}}{\data{arrow}} }
	\justifies
	{ \Gamma'' \vdash \data{lunitor}_{\bord{}{-}\data{arrow}}(\widehat{a}) }
	\using
	{\scriptstyle\data{lunitor}}
\]
}
\justifies
{
	\Gamma'', \, \mathit{lu}: ((\widehat{u} \cp{0} \widehat{a}) \cp{1} \widehat{m}) \celto \data{lunitor}_{\bord{}{-}\data{arrow}}(\widehat{a}) \vdash
}
\using
{\scriptstyle\data{gen}}
\end{prooftree}
\end{center}

\begin{flushright}
\begin{minipage}{0.7\textwidth}
\begin{lstlisting}[language=Python, firstnumber=7]
lu = Lun.add('lu', u.paste(a).paste(m), a.lunitor())
\end{lstlisting}
\end{minipage}
\end{flushright}

The following is a representation of $\textit{lu}$ as a term of $\diagset$, that is, an oriented graded poset labelled with names, together with string diagram representations of $\textit{lu}$, its input boundary, and its output boundary, and the $\code{rewalt}$ code that generated them.

\noindent\begin{minipage}{0.3\textwidth}
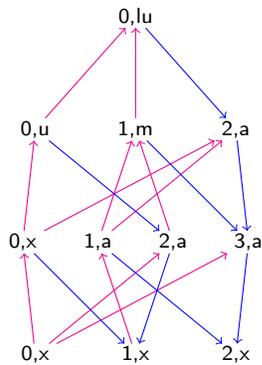
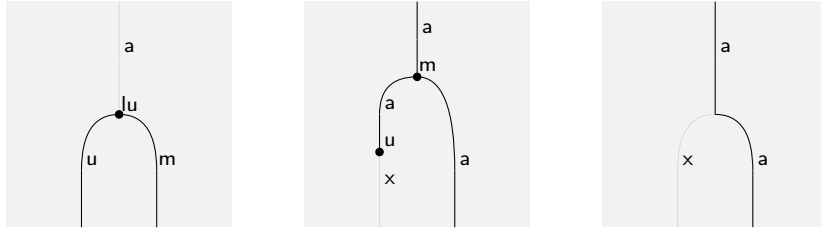
\begin{tikzpicture}[xscale=4, yscale=6, baseline={([yshift=-.5ex]current bounding box.center)}]
\path[fill=white] (0, 0) rectangle (1, 1);
\draw[->, draw=magenta] (0.1625, 0.15) -- (0.12916666666666665, 0.35);
\draw[->, draw=magenta] (0.2125, 0.15) -- (0.5791666666666667, 0.35);
\draw[->, draw=magenta] (0.2375, 0.15) -- (0.8041666666666667, 0.35);
\draw[->, draw=magenta] (0.4875, 0.15) -- (0.3875, 0.35);
\draw[->, draw=blue] (0.16249999999999998, 0.35) -- (0.4625, 0.15);
\draw[->, draw=magenta] (0.12916666666666665, 0.4) -- (0.1625, 0.6);
\draw[->, draw=magenta] (0.1958333333333333, 0.4) -- (0.7625, 0.6);
\draw[->, draw=blue] (0.42083333333333334, 0.35) -- (0.7875, 0.15);
\draw[->, draw=magenta] (0.3875, 0.4) -- (0.4875, 0.6);
\draw[->, draw=magenta] (0.42083333333333334, 0.4) -- (0.7875, 0.6);
\draw[->, draw=blue] (0.6125, 0.35) -- (0.5125, 0.15);
\draw[->, draw=magenta] (0.6125, 0.4) -- (0.5125, 0.6);
\draw[->, draw=blue] (0.8708333333333333, 0.35) -- (0.8374999999999999, 0.15);
\draw[->, draw=blue] (0.2125, 0.6) -- (0.5791666666666667, 0.4);
\draw[->, draw=magenta] (0.19999999999999998, 0.65) -- (0.4666666666666667, 0.85);
\draw[->, draw=blue] (0.5375, 0.6) -- (0.8375, 0.4);
\draw[->, draw=magenta] (0.5, 0.65) -- (0.5, 0.85);
\draw[->, draw=blue] (0.8374999999999999, 0.6) -- (0.8708333333333333, 0.4);
\draw[->, draw=blue] (0.5333333333333333, 0.85) -- (0.7999999999999999, 0.65);
\node[text=black, font={\scriptsize \sffamily}, xshift=0pt, yshift=0pt] at (0.16666666666666666, 0.125) {0,x};
\node[text=black, font={\scriptsize \sffamily}, xshift=0pt, yshift=0pt] at (0.5, 0.125) {1,x};
\node[text=black, font={\scriptsize \sffamily}, xshift=0pt, yshift=0pt] at (0.8333333333333333, 0.125) {2,x};
\node[text=black, font={\scriptsize \sffamily}, xshift=0pt, yshift=0pt] at (0.125, 0.375) {0,x};
\node[text=black, font={\scriptsize \sffamily}, xshift=0pt, yshift=0pt] at (0.375, 0.375) {1,a};
\node[text=black, font={\scriptsize \sffamily}, xshift=0pt, yshift=0pt] at (0.625, 0.375) {2,a};
\node[text=black, font={\scriptsize \sffamily}, xshift=0pt, yshift=0pt] at (0.875, 0.375) {3,a};
\node[text=black, font={\scriptsize \sffamily}, xshift=0pt, yshift=0pt] at (0.16666666666666666, 0.625) {0,u};
\node[text=black, font={\scriptsize \sffamily}, xshift=0pt, yshift=0pt] at (0.5, 0.625) {1,m};
\node[text=black, font={\scriptsize \sffamily}, xshift=0pt, yshift=0pt] at (0.8333333333333333, 0.625) {2,a};
\node[text=black, font={\scriptsize \sffamily}, xshift=0pt, yshift=0pt] at (0.5, 0.875) {0,lu};
\end{tikzpicture}
\end{minipage}%
\begin{minipage}{0.7\textwidth}
\begin{tikzpicture}[xscale=3, yscale=3, baseline={([yshift=-.5ex]current bounding box.center)}]
\path[fill=gray!10] (0, 0) rectangle (1, 1);
\path[fill, color=gray!10] (0.5, 0.5) .. controls (0.4933333333333333, 0.5) and (0.49, 0.5833333333333333) .. (0.49, 0.75) to (0.51, 0.75) .. controls (0.51, 0.5833333333333333) and (0.5066666666666667, 0.5) .. (0.5, 0.5);
\draw[color=black, opacity=0.1] (0.5, 0.5) .. controls (0.5, 0.5) and (0.5, 0.5833333333333333) .. (0.5, 0.75);
\path[fill, color=gray!10] (0.5, 1) .. controls (0.4933333333333333, 1.0) and (0.49, 0.9166666666666666) .. (0.49, 0.75) to (0.51, 0.75) .. controls (0.51, 0.9166666666666666) and (0.5066666666666667, 1.0) .. (0.5, 1);
\draw[color=black, opacity=0.1] (0.5, 1) .. controls (0.5, 1.0) and (0.5, 0.9166666666666666) .. (0.5, 0.75);
\path[fill, color=gray!10] (0.5, 0.5) .. controls (0.6044444444444445, 0.5) and (0.6566666666666666, 0.41666666666666663) .. (0.6566666666666666, 0.25) to (0.6766666666666666, 0.25) .. controls (0.6766666666666666, 0.41666666666666663) and (0.6177777777777778, 0.5) .. (0.5, 0.5);
\draw[color=black, opacity=1] (0.5, 0.5) .. controls (0.611111111111111, 0.5) and (0.6666666666666666, 0.41666666666666663) .. (0.6666666666666666, 0.25);
\path[fill, color=gray!10] (0.6666666666666666, 0) .. controls (0.6599999999999999, 0.0) and (0.6566666666666666, 0.08333333333333333) .. (0.6566666666666666, 0.25) to (0.6766666666666666, 0.25) .. controls (0.6766666666666666, 0.08333333333333333) and (0.6733333333333333, 0.0) .. (0.6666666666666666, 0);
\draw[color=black, opacity=1] (0.6666666666666666, 0) .. controls (0.6666666666666666, 0.0) and (0.6666666666666666, 0.08333333333333333) .. (0.6666666666666666, 0.25);
\path[fill, color=gray!10] (0.5, 0.5) .. controls (0.3822222222222222, 0.5) and (0.3233333333333333, 0.41666666666666663) .. (0.3233333333333333, 0.25) to (0.3433333333333333, 0.25) .. controls (0.3433333333333333, 0.41666666666666663) and (0.39555555555555555, 0.5) .. (0.5, 0.5);
\draw[color=black, opacity=1] (0.5, 0.5) .. controls (0.38888888888888884, 0.5) and (0.3333333333333333, 0.41666666666666663) .. (0.3333333333333333, 0.25);
\path[fill, color=gray!10] (0.3333333333333333, 0) .. controls (0.32666666666666666, 0.0) and (0.3233333333333333, 0.08333333333333333) .. (0.3233333333333333, 0.25) to (0.3433333333333333, 0.25) .. controls (0.3433333333333333, 0.08333333333333333) and (0.33999999999999997, 0.0) .. (0.3333333333333333, 0);
\draw[color=black, opacity=1] (0.3333333333333333, 0) .. controls (0.3333333333333333, 0.0) and (0.3333333333333333, 0.08333333333333333) .. (0.3333333333333333, 0.25);
\node[circle, fill=black, draw=black, inner sep=1pt] at (0.5, 0.5) {};
\node[text=black, font={\scriptsize \sffamily}, xshift=4pt, yshift=4pt] at (0.5, 0.75) {a};
\node[text=black, font={\scriptsize \sffamily}, xshift=4pt, yshift=4pt] at (0.6666666666666666, 0.25) {m};
\node[text=black, font={\scriptsize \sffamily}, xshift=4pt, yshift=4pt] at (0.3333333333333333, 0.25) {u};
\node[text=black, font={\scriptsize \sffamily}, xshift=4pt, yshift=4pt] at (0.5, 0.5) {lu};
\end{tikzpicture}
\quad \quad
\begin{tikzpicture}[xscale=3, yscale=3, baseline={([yshift=-.5ex]current bounding box.center)}]
\path[fill=gray!10] (0, 0) rectangle (1, 1);
\path[fill, color=gray!10] (0.5, 0.6666666666666666) .. controls (0.4933333333333333, 0.6666666666666666) and (0.49, 0.7222222222222222) .. (0.49, 0.8333333333333334) to (0.51, 0.8333333333333334) .. controls (0.51, 0.7222222222222222) and (0.5066666666666667, 0.6666666666666666) .. (0.5, 0.6666666666666666);
\draw[color=black, opacity=1] (0.5, 0.6666666666666666) .. controls (0.5, 0.6666666666666666) and (0.5, 0.7222222222222222) .. (0.5, 0.8333333333333334);
\path[fill, color=gray!10] (0.5, 1) .. controls (0.4933333333333333, 1.0) and (0.49, 0.9444444444444444) .. (0.49, 0.8333333333333334) to (0.51, 0.8333333333333334) .. controls (0.51, 0.9444444444444444) and (0.5066666666666667, 1.0) .. (0.5, 1);
\draw[color=black, opacity=1] (0.5, 1) .. controls (0.5, 1.0) and (0.5, 0.9444444444444444) .. (0.5, 0.8333333333333334);
\path[fill, color=gray!10] (0.3333333333333333, 0.3333333333333333) .. controls (0.32666666666666666, 0.3333333333333333) and (0.3233333333333333, 0.38888888888888884) .. (0.3233333333333333, 0.5) to (0.3433333333333333, 0.5) .. controls (0.3433333333333333, 0.38888888888888884) and (0.33999999999999997, 0.3333333333333333) .. (0.3333333333333333, 0.3333333333333333);
\draw[color=black, opacity=1] (0.3333333333333333, 0.3333333333333333) .. controls (0.3333333333333333, 0.3333333333333333) and (0.3333333333333333, 0.38888888888888884) .. (0.3333333333333333, 0.5);
\path[fill, color=gray!10] (0.5, 0.6666666666666666) .. controls (0.3822222222222222, 0.6666666666666666) and (0.3233333333333333, 0.611111111111111) .. (0.3233333333333333, 0.5) to (0.3433333333333333, 0.5) .. controls (0.3433333333333333, 0.611111111111111) and (0.39555555555555555, 0.6666666666666666) .. (0.5, 0.6666666666666666);
\draw[color=black, opacity=1] (0.5, 0.6666666666666666) .. controls (0.38888888888888884, 0.6666666666666666) and (0.3333333333333333, 0.611111111111111) .. (0.3333333333333333, 0.5);
\path[fill, color=gray!10] (0.5, 0.6666666666666666) .. controls (0.6044444444444445, 0.6666666666666666) and (0.6566666666666666, 0.5277777777777778) .. (0.6566666666666666, 0.25) to (0.6766666666666666, 0.25) .. controls (0.6766666666666666, 0.5277777777777778) and (0.6177777777777778, 0.6666666666666666) .. (0.5, 0.6666666666666666);
\draw[color=black, opacity=1] (0.5, 0.6666666666666666) .. controls (0.611111111111111, 0.6666666666666666) and (0.6666666666666666, 0.5277777777777778) .. (0.6666666666666666, 0.25);
\path[fill, color=gray!10] (0.6666666666666666, 0) .. controls (0.6599999999999999, 0.0) and (0.6566666666666666, 0.08333333333333333) .. (0.6566666666666666, 0.25) to (0.6766666666666666, 0.25) .. controls (0.6766666666666666, 0.08333333333333333) and (0.6733333333333333, 0.0) .. (0.6666666666666666, 0);
\draw[color=black, opacity=1] (0.6666666666666666, 0) .. controls (0.6666666666666666, 0.0) and (0.6666666666666666, 0.08333333333333333) .. (0.6666666666666666, 0.25);
\path[fill, color=gray!10] (0.3333333333333333, 0.3333333333333333) .. controls (0.32666666666666666, 0.3333333333333333) and (0.3233333333333333, 0.2777777777777778) .. (0.3233333333333333, 0.16666666666666666) to (0.3433333333333333, 0.16666666666666666) .. controls (0.3433333333333333, 0.2777777777777778) and (0.33999999999999997, 0.3333333333333333) .. (0.3333333333333333, 0.3333333333333333);
\draw[color=black, opacity=0.1] (0.3333333333333333, 0.3333333333333333) .. controls (0.3333333333333333, 0.3333333333333333) and (0.3333333333333333, 0.2777777777777778) .. (0.3333333333333333, 0.16666666666666666);
\path[fill, color=gray!10] (0.3333333333333333, 0) .. controls (0.32666666666666666, 0.0) and (0.3233333333333333, 0.05555555555555555) .. (0.3233333333333333, 0.16666666666666666) to (0.3433333333333333, 0.16666666666666666) .. controls (0.3433333333333333, 0.05555555555555555) and (0.33999999999999997, 0.0) .. (0.3333333333333333, 0);
\draw[color=black, opacity=0.1] (0.3333333333333333, 0) .. controls (0.3333333333333333, 0.0) and (0.3333333333333333, 0.05555555555555555) .. (0.3333333333333333, 0.16666666666666666);
\node[circle, fill=black, draw=black, inner sep=1pt] at (0.3333333333333333, 0.3333333333333333) {};
\node[circle, fill=black, draw=black, inner sep=1pt] at (0.5, 0.6666666666666666) {};
\node[text=black, font={\scriptsize \sffamily}, xshift=4pt, yshift=4pt] at (0.5, 0.8333333333333334) {a};
\node[text=black, font={\scriptsize \sffamily}, xshift=4pt, yshift=4pt] at (0.3333333333333333, 0.5) {a};
\node[text=black, font={\scriptsize \sffamily}, xshift=4pt, yshift=4pt] at (0.6666666666666666, 0.25) {a};
\node[text=black, font={\scriptsize \sffamily}, xshift=4pt, yshift=4pt] at (0.3333333333333333, 0.16666666666666666) {x};
\node[text=black, font={\scriptsize \sffamily}, xshift=4pt, yshift=4pt] at (0.3333333333333333, 0.3333333333333333) {u};
\node[text=black, font={\scriptsize \sffamily}, xshift=4pt, yshift=4pt] at (0.5, 0.6666666666666666) {m};
\end{tikzpicture}
\quad \quad
\begin{tikzpicture}[xscale=3, yscale=3, baseline={([yshift=-.5ex]current bounding box.center)}]
\path[fill=gray!10] (0, 0) rectangle (1, 1);
\path[fill, color=gray!10] (0.5, 0.5) .. controls (0.4933333333333333, 0.5) and (0.49, 0.5833333333333333) .. (0.49, 0.75) to (0.51, 0.75) .. controls (0.51, 0.5833333333333333) and (0.5066666666666667, 0.5) .. (0.5, 0.5);
\draw[color=black, opacity=1] (0.5, 0.5) .. controls (0.5, 0.5) and (0.5, 0.5833333333333333) .. (0.5, 0.75);
\path[fill, color=gray!10] (0.5, 1) .. controls (0.4933333333333333, 1.0) and (0.49, 0.9166666666666666) .. (0.49, 0.75) to (0.51, 0.75) .. controls (0.51, 0.9166666666666666) and (0.5066666666666667, 1.0) .. (0.5, 1);
\draw[color=black, opacity=1] (0.5, 1) .. controls (0.5, 1.0) and (0.5, 0.9166666666666666) .. (0.5, 0.75);
\path[fill, color=gray!10] (0.5, 0.5) .. controls (0.6044444444444445, 0.5) and (0.6566666666666666, 0.41666666666666663) .. (0.6566666666666666, 0.25) to (0.6766666666666666, 0.25) .. controls (0.6766666666666666, 0.41666666666666663) and (0.6177777777777778, 0.5) .. (0.5, 0.5);
\draw[color=black, opacity=1] (0.5, 0.5) .. controls (0.611111111111111, 0.5) and (0.6666666666666666, 0.41666666666666663) .. (0.6666666666666666, 0.25);
\path[fill, color=gray!10] (0.6666666666666666, 0) .. controls (0.6599999999999999, 0.0) and (0.6566666666666666, 0.08333333333333333) .. (0.6566666666666666, 0.25) to (0.6766666666666666, 0.25) .. controls (0.6766666666666666, 0.08333333333333333) and (0.6733333333333333, 0.0) .. (0.6666666666666666, 0);
\draw[color=black, opacity=1] (0.6666666666666666, 0) .. controls (0.6666666666666666, 0.0) and (0.6666666666666666, 0.08333333333333333) .. (0.6666666666666666, 0.25);
\path[fill, color=gray!10] (0.5, 0.5) .. controls (0.3822222222222222, 0.5) and (0.3233333333333333, 0.41666666666666663) .. (0.3233333333333333, 0.25) to (0.3433333333333333, 0.25) .. controls (0.3433333333333333, 0.41666666666666663) and (0.39555555555555555, 0.5) .. (0.5, 0.5);
\draw[color=black, opacity=0.1] (0.5, 0.5) .. controls (0.38888888888888884, 0.5) and (0.3333333333333333, 0.41666666666666663) .. (0.3333333333333333, 0.25);
\path[fill, color=gray!10] (0.3333333333333333, 0) .. controls (0.32666666666666666, 0.0) and (0.3233333333333333, 0.08333333333333333) .. (0.3233333333333333, 0.25) to (0.3433333333333333, 0.25) .. controls (0.3433333333333333, 0.08333333333333333) and (0.33999999999999997, 0.0) .. (0.3333333333333333, 0);
\draw[color=black, opacity=0.1] (0.3333333333333333, 0) .. controls (0.3333333333333333, 0.0) and (0.3333333333333333, 0.08333333333333333) .. (0.3333333333333333, 0.25);
\node[text=black, font={\scriptsize \sffamily}, xshift=4pt, yshift=4pt] at (0.5, 0.75) {a};
\node[text=black, font={\scriptsize \sffamily}, xshift=4pt, yshift=4pt] at (0.6666666666666666, 0.25) {a};
\node[text=black, font={\scriptsize \sffamily}, xshift=4pt, yshift=4pt] at (0.3333333333333333, 0.25) {x};
\end{tikzpicture}

\begin{lstlisting}[language=Python, firstnumber=8]
lu.hasse(tikz=True)
lu.draw(bgcolor='gray!10', tikz=True)
lu.input.draw(bgcolor='gray!10', tikz=True)
lu.output.draw(bgcolor='gray!10', tikz=True)
\end{lstlisting}
\end{minipage}
\end{exm}

\begin{comm}
Provided we have a unique underlying representation of shapes, as described in Section \ref{sec:uniquerep}, every term of $\diagset$ also has a unique representation.
In this sense, terms of $\diagset$ are ``noncomputational'': all the computation, which consists exclusively of computing and matching shapes, happens under the hood before a term is even created, so the equality theory of terms is trivial.

This is intended. Rather than a computational theory in itself, $\diagset$ is intended as a \emph{substrate for computational theories} according to the paradigm of higher-dimensional rewriting.
A term $t: r^- \celto r^+$ can be seen as a rewrite of the ``lower-dimensional'' term $r^-$ to the term $r^+$, and the extension of $t$ via the $\data{paste}_k$ rules establishes how the rewrite can happen in a wider context. In this sense, every well\nbd formed context in $\diagset$ contains its own internal computational theory on terms of each dimension.
\end{comm}

\begin{rmk}
While ``rewrites in context'' can be built with the $\data{paste}_k$ rules, this is quite impractical.
In practice, one wants to start from a diagram and apply a generating rewrite directly to a subdiagram.
This is modelled by \emph{pasting along a subdiagram} \cite[\S 4.12]{hadzihasanovic2020diagrammatic} in the theory of diagrammatic sets.

Pasting along a subdiagram is implemented in $\code{rewalt}$ with methods $\code{to\_inputs}$ and $\code{to\_outputs}$. 
These invoke a procedure for recognising subdiagrams, which currently uses a quite naive algorithm.
The issue of recognising subdiagrams deserves further study, so we leave it to future work. 
\end{rmk}

\section*{Conclusions and outlook}

We have provided a formal implementation of ``plain'' diagrammatic sets. 
An obvious next step is the formalisation of \emph{weakly invertible} cells, and then of diagrammatic sets with weak composites, a model of weak higher categories \cite[Sections 5, 6]{hadzihasanovic2020diagrammatic}.
This is in fact part of \code{rewalt}, but still lacks a formal analysis.

In addition, we still have a limited range of high-level methods for handling weak units.
We may want, for example, flexible higher-dimensional versions of ``Mac Lane triangle'' rules for shuffling weak units around.
Development of these methods, and others tailored to specific applications, will likely go hand in hand with practical experience in the use of \code{rewalt} as a proof assistant.

To conclude, we have only scratched the surface of the algorithm and complexity theory of diagram rewriting in higher dimensions.
In particular, we have not yet studied the problem of searching for a subdiagram within another diagram, whose solution is essential to any form of fully automated or assisted diagram rewriting.
We plan to tackle this problem in future work.

\nocite{*}
\bibliographystyle{eptcs}
\bibliography{main}

\end{document}